\documentclass[12pt,reqno,a4paper]{article}
\usepackage{mathtools}
\usepackage{enumerate}
\usepackage{comment}
\usepackage{graphicx} 
\usepackage{amsmath,amssymb,amsfonts,amsthm}
\usepackage{xcolor}
\usepackage[
inner=2.3cm,
outer=2.3cm,
top=2.7cm,
marginparwidth=2.5cm,
marginparsep=0.1cm
]
{geometry}
\allowdisplaybreaks

\newcommand{\ve}{\varepsilon}
\newcommand\cL{{\mathcal L}}
\newcommand\bP{{\mathbb P}}
\newcommand\cC{{\mathcal C}}

\newtheoremstyle{special}%
{}%
{}%
{}%
{}%
{\scshape}%
{.}%
{.5em}%
{}
\newtheorem{theorem}{Theorem}

\newtheorem{lemma}[theorem]{Lemma}
\newtheorem{proposition}[theorem]{Proposition}
\newtheorem{remark}[theorem]{Remark}

\def\Id{\text{\rm Id}}

\DeclareMathOperator{\esssup}{esssup}
\DeclareMathOperator{\essinf}{essinf}
\begin{document}
	
	\title{Regularity of the variance in quenched CLT for random intermittent dynamical systems}
	\author{Davor Dragi\v cevi\' c
		\thanks{Faculty of Mathematics, University of Rijeka, Radmile Matej\v ci\' c 2, 51000 Rijeka, Croatia.\\ Email: ddragicevic@math.uniri.hr}
		\and
		Juho Lepp{\"a}nen \thanks{Department of Mathematics, Tokai University, Kanagawa 259-1292, Japan\\ Email: leppanen.juho.heikki.g@tokai.ac.jp}
	}

	\maketitle
	
	\begin{abstract}
		We study random dynamical systems composed of Liverani--Saussol--Vaienti maps with varying parameters, without any mixing assumptions on the base space of random dynamics. We establish a quenched central limit theorem and identify conditions under which the associated limit variance varies continuously and differentiably with respect to perturbations of the random dynamics. Our arguments rely on recent results on statistical stability and linear response for random intermittent maps established in Dragičević et al. (J. Lond. Math. Soc. 111 (2025), e70150).
	\end{abstract}

	\vspace{3mm}
	
	\vspace{2mm}
	
	
	\section{Introduction}
	
	In their seminal paper~\cite{LSV}, Liverani, Saussol and Vaienti introduced a now famous class of intermittent maps (called LSV maps) on the unit interval. 
	Trajectories of these nonuniformly expanding maps behave chaotically on a large part of
	the state space, but when they visit small neighborhoods of the neutral fixed point at the origin, it
	takes a long time before they return to the chaotic regime. We stress that the study of statistical properties of such maps turned out to be a challenging problem that resulted in the development of new techniques. For example, it turned out that such maps exhibit only polynomial decay of correlations, while uniformly expanding (or hyperbolic) maps possess exponential decay of correlations.  We refer to~\cite{alvesbook} for a detailed survey.
	
	In recent years, there has been a growing interest in the study of statistical properties of nonstationary (or nonautonomous) dynamical systems formed by sequential or random compositions of LSV maps with varying parameters. Notable contributions to their study include estimates on memory loss and decay of correlations (see, for example,~\cite{AHNTV, BBD, BBR, KL,KL2, leppanen2017functional}), various limit laws (see~\cite{ANT, BB, bose2021random, HL,  NPT, NTV,  Su, Su2} and references therein) and linear response~\cite{bahsoun2020linear, DGTS}.
	
	In this paper, we first establish the quenched central limit theorem (CLT) for random compositions of LSV maps 
	with varying parameters. Previous results on the quenched CLT in this setting include~\cite{NPT, NTV} for 
	i.i.d. compositions 
	and ~\cite{Su2} for ergodic compositions. Our theorem (see Theorem~\ref{qclt}), which applies to ergodic compositions, also shows
	an ``unexpected'' CLT in the spirit of  Gou\"{e}zel \cite{G}: for a class of special observables that vanish 
	at the neutral fixed point, the CLT holds in a parameter range where it is not expected for a 
	generic regular observable.
	We deduce the quenched CLT from an abstract result
	due to Kifer~\cite[Theorem 2.3.]{K} which, to our knowledge, has not previously been applied in this context. This result can be seen as a natural development in the recent line of research devoted to quenched limit theorems for broad classes of random dynamical systems that avoid any mixing assumptions for the base system (see~\cite{DFGTV, DFGTV2, DH, HK}).
	
	Our main objective is to study the regularity of the limit variance in Theorem~\ref{qclt} under perturbations of the random intermittent dynamics. More precisely, Theorem~\ref{contvariance} (resp. Theorem~\ref{diffvar}) gives conditions under which the variance varies continuously (resp. differentiably) under suitable perturbations. 
	Although both results consider perturbations of the same type, Theorem~\ref{contvariance} applies within the parameter range where correlations are summable, similar to the CLT in Theorem~\ref{qclt}, whereas Theorem~\ref{diffvar} requires a more restrictive parameter range.
	Consequently, Theorem~\ref{contvariance} is not implied by Theorem~\ref{diffvar}.
	
	We note that there exist several results in the literature devoted to the regularity of variance in the CLT for various classes of deterministic dynamics~\cite{BMR, BCV, GKLM, Selley}, among which the recent work of S\'elley~\cite{Selley} dealing with LSV maps is closest to the setting of our paper. In the context of random dynamics, the regularity of the variance in the quenched CLT was studied for random uniformly hyperbolic dynamics~\cite{DS}, certain examples of random partially hyperbolic systems~\cite{CN}, and for broad classes of nonuniformly 
	expanding random systems~\cite{DH2}, which do not include random LSV maps (see below for more details).
	
	Consequently, Theorem~\ref{diffvar} can be seen as a random counterpart of the main result of~\cite{Selley}, although we rely on completely different techniques. More precisely, the arguments in~\cite{Selley} rely on inducing, while our arguments are based on the recent statistical stability and linear response results for random intermittent maps established in~\cite{DGTS}, together with the arguments developed in the proof of~\cite[Theorem 12]{DH2}. However, we stress that there are nontrivial differences between the proof of~\cite[Theorem 12]{DH2} and Theorem~\ref{diffvar}, as~\cite[Theorem 12]{DH2} requires (and its proof uses) that the assumptions of~\cite[Theorem 2]{DH2} hold with $\mathcal B_w=C^0$, $\mathcal B_s=C^1$ and $\mathcal B_{ss}=C^3$. In our setting, \cite[(4)]{DH2} holds with $\mathcal B_w=L^1(m)$ (see~\eqref{dec} for $h=1$). The same comment applies for other conditions in~\cite[Theorem 2]{DH2}.

	Thus, it is necessary to combine the approach from~\cite{DH2} with the cone techniques used in the study of LSV maps initiated in~\cite{LSV}.
	
	We emphasize that the proof of Theorem~\ref{diffvar} actually requires a certain improvement of the linear response result of~\cite{DGTS} (see Theorem~\ref{sharpl}
	and Remark~\ref{sharplc}) and that contrary to~\cite{Selley} it yields an explicit formula for the derivative of the variance (see Remark~\ref{formula}). On the other hand, Theorem~\ref{diffvar} does require (in comparison with Theorems~\ref{qclt} and~\ref{contvariance}) an additional constraint for the range of parameter values of the LSV maps. Unfortunately, we currently do not see a possible way to eliminate this constraint. We refer the 
	reader to Remark \ref{formula} for further discussion on this issue.

	\paragraph{\textbf{Convention.}} Throughout this work $C$ and $D$ will denote positive constants whose precise values may change from one appearance to the next. Dependence on parameters such as $\alpha$ will be indicated with subscripts, e.g. $C=C_\alpha$, as needed.
	\section{Quenched central limit theorem for random LSV maps}\label{sec:clt}
	
	Throughout this paper, $(\Omega, \mathcal F, \mathbb P)$ is an arbitrary probability space equipped with an invertible $\mathbb P$-preserving transformation $\sigma \colon \Omega \to \Omega$. We assume that $\mathbb P$ is ergodic. Moreover, $m$ will denote the Lebesgue measure on the unit interval $[0, 1]$. We recall that for each
	$\gamma \in (0, 1)$ the associated Liverani-Saussol-Vaienti (LSV) map (introduced in~\cite{LSV}) is given by 
	\[
	T_\gamma(x)=\begin{cases}
		x(1+2^\gamma x^\gamma) & 0\le x<\frac 1 2 \\
		2x-1 & \frac 1 2 \le x\le 1.
	\end{cases}
	\]
	Moreover, for $\gamma \in (0, 1)$ and an observable $\phi$ we set
	\[
	N_\gamma\phi:= g_\gamma'\phi\circ g_\gamma,
	\]
	where $g_\gamma$ denotes the inverse of the restriction of $T_\gamma$ to $[0, \frac 1 2]$. Finally,  we will also denote $v_\gamma(x):=\partial_\gamma T_\gamma(x)=2^\gamma x^{1+\gamma}\log(2x)$ and \[
	X_\gamma:=v_\gamma\circ g_\gamma=2^\gamma g_\gamma^{1+\gamma}\log(2g_\gamma).
	\]
	
	For each $\omega \in \Omega$, let $T_\omega  \colon [0, 1] \to [0, 1]$ be an LSV map with the parameter $\beta (\omega) \in (0, 1)$.  As usual, we set 
	\[
	T_\omega^n:=T_{\sigma^{n-1}\omega}\circ \ldots \circ T_{\sigma \omega}\circ T_\omega, \quad \omega \in \Omega,  \ n\in \mathbb N.
	\]
	In the same manner, denoting by $\mathcal L_\omega \colon L^1(m)\to L^1(m)$ the transfer operator associated to $T_\omega$ and $m$, we set 
	\[
	\mathcal L_\omega^n:=\mathcal L_{\sigma^{n-1}\omega}\circ \ldots \circ \mathcal L_{\sigma \omega}\circ \mathcal L_\omega, \quad \omega \in \Omega,  \ n\in \mathbb N.
	\]
	We also set $\mathcal L_\omega^0 :=\Id$ for $\omega \in \Omega$.
	In the sequel, we will assume that $\omega \mapsto \beta (\omega)$ is measurable and that 
	\[
	\alpha:=\esssup_{\omega \in \Omega}\beta(\omega)< 1.
	\]
	It is proved in~\cite[Proposition 9]{DGTS} that the cocycle of maps $(T_\omega)_{\omega \in \Omega}$ admits a unique random a.c.i.m $\mu$ on $\Omega \times [0, 1]$, which can be identified with a family of probability measures $(\mu_\omega)_{\omega \in \Omega}$ on $[0, 1]$ such that
	\[
	T_\omega^*\mu_\omega=\mu_{\sigma \omega} \quad \text{for $\mathbb P$-a.e. $\omega \in \Omega$.}
	\]
	Moreover, $d\mu_\omega=h_\omega \, dm$ with $h_\omega \in \mathcal C_*\cap \mathcal C_2$ for some $a>1$, where $\mathcal C_*=\mathcal C_*(a)$ and $\mathcal C_2=\mathcal C_2(b_1, b_2)$ are cones as in~\cite[Section 2.2]{DGTS}, that is $\mathcal C_*$ consist of  $\phi \in C^0(0, 1]\cap L^1(m)$ such that $\phi \ge 0$, $\phi$ decreasing, $X^{\alpha+1}\phi$ increasing (where $X$ denotes the identity map), and 
	\[
	\int_0^x \phi (t)\, dt\le ax^{1-\alpha}\int_0^1 \phi (t)\, dt \quad x\in (0, 1].
	\]
	Moreover, $\mathcal C_2$ consists of all $\phi \in C^2(0, 1]$ so that
	\[
	\phi (x)\ge 0, \quad |\phi'(x)| \le \frac{b_1}{x}\phi(x) \quad \text{and} \quad |\phi''(x)| \le \frac{b_2}{x^2}\phi(x), \quad x\in (0, 1].
	\]
	We stress that the parameters $a$, $b_1$ and $b_2$ depend only on $\alpha$. Throughout the paper, we will often use the following simple result.
	
	If $h\in \mathcal C_*$ and $\varphi \in C^1[0, 1]$, then by 
	\cite[Lemma 3.3]{LS} (see also \cite[Remark 2.5]{NTV} and \cite[Lemma 3.4]{NPT}), 
	\begin{align}\label{eq:decomp_c1}
		\varphi h = g_1 - g_2 
	\end{align}
	for $g_1,g_2 \in \cC_*$ with $\Vert g_i \Vert_{L^1(m)} \le C_\alpha \Vert \varphi \Vert_{C^1} \Vert h \Vert_{L^1(m)}$ for $i = 1,2$. Moreover, if $h \in C(0,1]$, then $g_i \in C^1(0,1]$.
	
	\begin{proposition}
		There exists $C_\alpha>0$ such that 
		\begin{equation}\label{dec}
			\|\mathcal L_\omega^n (\varphi h)\|_{L^1(m)} \le C_\alpha n^{-1/\alpha+1}\|\varphi \|_{C^1} \cdot \|h\|_{L^1(m)},
		\end{equation}
		for $\mathbb P$-a.e. $\omega \in \Omega$,  $n\in \mathbb N$, and $\varphi \in C^1[0, 1]$, $h\in \mathcal C_*$ with
		$\int_0^1 \varphi h\, dm=0$.
	\end{proposition}
	
	\begin{proof} 
		The result follows by combining 
		\eqref{DECC}
		with \eqref{eq:decomp_c1}.
	\end{proof}
	
	\begin{proposition} Suppose that $\psi \in C^1(0,1]$ satisfies 
		$\int_0^1 \psi \, d m = 0$ and, for $x \in (0,1]$,
		$$
		|\psi(x)| \le C_{0} x^{-\gamma} \quad 
		\text{and} \quad 
		|\psi'(x)| \le C_{1} x^{-\gamma - 1},
		$$
		where $0 < \gamma \le \alpha$ and $C_0,C_1 > 0$.
		Then, there exists $C=C_{a, \alpha, \gamma}>0$
		such that, for $\bP$-a.e. $\omega \in \Omega$, 
		\begin{align}\label{eq:fast_ml_1}
			\Vert \cL_\omega^n (  \psi ) \Vert_{L^1(m)} \le C(C_{0} + C_{ 1})
			n^{-1/\alpha + \gamma/ \alpha}
		\end{align}
		holds for $n \in \mathbb{N}$. In particular, if $\varphi \in C^1[0, 1]$ satisfies
		$\int_0^1 \varphi h_\omega \, dm = 0$ for $\bP$-a.e. $\omega \in \Omega$, 
		and 
		$$
		|\varphi(x)| \le C_\varphi x^{\gamma} \quad \forall x \in [0,1], 
		$$
		where $\gamma \ge 0$, 
		then there exists a constant 
		$C=C_{a, \alpha, \gamma} > 0$ 
		such that, for $\bP$-a.e. $\omega \in \Omega$, 
		\begin{align}\label{eq:fast_ml_2}
			\Vert \cL_\omega^n ( \varphi h_\omega  ) \Vert_{L^1(m)} \le C ( C_\varphi + \Vert \varphi \Vert_{C^1}  ) n^{ 1 - (1 +  \min\{ \gamma, \alpha \} )/\alpha}
		\end{align}
		holds for $n \in \mathbb{N}$.
	\end{proposition}

	\begin{proof} Denote by $\cC_*(a; \gamma)$ the cone 
		obtained by replacing $\alpha$ with $\gamma$ in the definition of $\mathcal C_*$.
		It follows from \cite[Lemma 3]{DGTS} that\footnote{Note that
			although $F \in C^2(0,1]$ is assumed in \cite[Lemma 3]{DGTS}, 
			the proof only uses $F \in C^1(0,1]$.
		}
		$$
		\psi(x) = \psi(x) + \lambda x^{- \gamma} -  \lambda x^{- \gamma} = \psi_{1} - \psi_{ 2},
		$$
		where $\psi_{i} \in \cC_*(a; \gamma)$, provided that 
		$$
		\lambda \ge ( C_{0} + C_{1}  ) \max \biggl\{ 
		2, \frac{1}{\gamma}, \frac{1}{a - 1}
		\biggr\}.
		$$
		If $\int_0^1 \psi_{ 1} \, dm = 0$, then $\psi_1(x) = 0$ for $m$-a.e. $x \in [0,1]$ and 
		\eqref{eq:fast_ml_1} holds trivially. Otherwise, 
		let $\tilde{\nu}_{i}$ denote the probability measure with density $\tilde{\psi}_{ i} = \psi_{ i} / \int_0^1 \psi_{ i} \, dm \in \cC_*(a; \gamma)$.
		Then, $\tilde{\nu}_{i}$ is regular in the sense 
		of \cite[Definition 3.5]{KL}; see \cite[Proposition 3.14]{KL}. Moreover, the first entry time
		$$
		\tau_\omega(x) = \inf \{  n \ge 1 \: : \:  T_\omega^n(x)  \in [1/2,1] \}
		$$
		satisfies 
		$$
		\tilde{\nu}_{i}( \tau_\omega \ge n ) \le C_{a, \gamma, \alpha} n^{- 1 / \alpha + \gamma / \alpha}.
		$$
		In the terminology of \cite{KL}, this means that $\tilde{\nu}_{i}$ has tail bound 
		$C_{a,\gamma, \alpha} n^{- 1 / \alpha + \gamma / \alpha}$. Note that
		$$
		\Vert \cL_\omega^n (  \psi ) \Vert_{L^1(m)}
		=  \frac{\lambda}{ 1 -\gamma} \Vert \cL_\omega^n (   \tilde{\psi}_{ 1} -    \tilde{\psi}_{ 2}   ) \Vert_{L^1(m)}.
		$$
		Now, \eqref{eq:fast_ml_1} follows  from \cite[Theorem 3.8]{KL}, and 
		\eqref{eq:fast_ml_2} is a straightforward consequence 
		of \eqref{eq:fast_ml_1} using $h_\omega \in \cC_* \cap \cC_2$.
	\end{proof}

	We are now in a position to state and prove our quenched central limit theorem for random compositions of LSV maps.
	\begin{theorem}\label{qclt}
		Let $\Phi \colon \Omega \times [0, 1] \to \mathbb R$ be a measurable map such that $\varphi_\omega:=\Phi (\omega, \cdot)\in C^1[0, 1]$ for $\omega \in \Omega$ and that 
		\begin{equation}\label{integ}
			\omega \mapsto \|\varphi_\omega\|_{C^1} \in L^2(\Omega, \mathcal F, \mathbb P).
		\end{equation}
		Moreover, assume that one of the following two conditions holds:
		\begin{itemize}
			\item[(i)] $\alpha < 1/2$.
			\item[(ii)] $\alpha \ge 1/2$ and,
			for $x \in [0,1]$ and $\bP$-a.e. $\omega \in \Omega$, 
			\begin{equation}\label{special_observable}
				|\varphi_\omega(x)| \le C_\omega x^{\gamma} \quad 
				\text{and} \quad \int_0^1 \varphi_\omega h_\omega \, dm = 0
			\end{equation}
			hold with $\gamma > 2 \alpha - 1$, where $\omega \mapsto C_\omega \in L^2(\Omega, \mathcal{F}, \bP)$.
		\end{itemize}
		Then, for $\mathbb P$-a.e. $\omega \in \Omega$,
		\[
		\frac{\sum_{i=0}^{n-1}(\varphi_{\sigma^i \omega}\circ T_\omega^i-\int_0^1\varphi_{\sigma^i \omega}\, d\mu_{\sigma^i \omega})}{\sqrt n}\to_d \mathcal N(0, \Sigma^2) \quad \text{on $([0, 1], \mu_\omega)$,}
		\]
		where 
		\[
		\Sigma^2:=\int_\Omega \int_0^1 \varphi_{\omega, 0}^2h_\omega \, dm\, d\mathbb P(\omega)+2\sum_{n=1}^\infty\int_\Omega \int_0^1 (\varphi_{\omega, 0}h_\omega) \cdot (\varphi_{\sigma^n \omega, 0}\circ T_\omega^n)\, dm\, d\mathbb P(\omega),
		\]
		and $\varphi_{\omega, 0}:=\varphi_\omega-\int_0^1 \varphi_\omega h_\omega \, dm$. Here $\mathcal N(0, \Sigma^2)$ denotes the normal distribution with mean $0$ and variance $\Sigma^2$ provided that $\Sigma^2>0$, and the unit mass in $0$ in the case $\Sigma^2=0$.
	\end{theorem}
	
	\begin{proof} First, we assume that (i) holds.
		We will verify the assumptions of~\cite[Theorem 2.3.]{K} with $Q=\Omega$ (i.e., when there is no inducing involved) and with the fiberwise centered observable.
		Note that $\varphi_{\omega, 0} \in C^1[0, 1]$ for $\omega \in \Omega$.
		Since 
		\[
		\left |\int_0^1\varphi_\omega  \, d\mu_\omega\right |=\left | \int_0^1\varphi_\omega h_\omega \, dm \right |\le \int_0^1|\varphi_\omega| h_\omega \, dm\le \|\varphi_\omega \|_{C^1}\int_0^1h_\omega \, dm=\|\varphi_\omega\|_{C^1},
		\]
		we have that $\|\varphi_{\omega, 0}\|_{C^1}\le 2\|\varphi_\omega \|_{C^1}$, and thus from~\eqref{integ} we get that 
		\begin{equation}\label{integ2}
			\omega \mapsto \|\varphi_{\omega, 0}\|_{C^1} \in L^2(\Omega, \mathcal F, \mathbb P).
		\end{equation}
		Observing that $(\mathbb E_{\mu_\omega}|\varphi_{\omega, 0}|^2)^{1/2}\le \|\varphi_{\omega, 0}\|_{C^1}$, we see that~\eqref{integ2} implies~\cite[(2.7)]{K}.
		
		Next, we aim to verify that~\cite[Theorem 2.3 (i)]{K} is true. Using~\eqref{dec} we have that  
		\begin{equation}\label{eq:cond_1}
			\begin{split}
				\sum_{n=1}^\infty \left |\mathbb E_{\mu_\omega}(\varphi_{\omega, 0} \cdot (\varphi_{\sigma^n \omega, 0}\circ T_\omega^n))\right | &=\sum_{n=1}^\infty \left |\int_0^1 \varphi_{\omega, 0} h_\omega (\varphi_{\sigma^n \omega, 0}\circ T_\omega^n)\, dm \right | \\
				&=\sum_{n=1}^\infty \left |\int_0^1 \varphi_{\sigma^n \omega, 0}  \mathcal L_\omega^n( \varphi_{\omega, 0} h_\omega) \, dm \right |  \\
				&\le \sum_{n=1}^\infty \|\mathcal L_\omega^n (\varphi_{\omega, 0} h_\omega)\|_{L^1(m)} \cdot \|\varphi_{\sigma^n \omega, 0}\|_{C^1} \\
				&\le C_\alpha\sum_{n=1}^\infty n^{-1/\alpha+1}\|\varphi_{\omega, 0} \|_{C^1} \cdot \|h_\omega\|_{L^1(m)}\cdot \|\varphi_{\sigma^n \omega, 0}\|_{C^1} \\
				&= C_\alpha\sum_{n=1}^\infty n^{-1/\alpha+1}\|\varphi_{\omega, 0}\|_{C^1}\cdot \|\varphi_{\sigma^n \omega, 0}\|_{C^1}.
			\end{split}
		\end{equation}
		Consequently, using H\"{o}lder's inequality and that $\sigma$ preserves $\mathbb P$,  we have that 
		\[
		\begin{split}
			\mathbb E_{\mathbb P}\sum_{n=1}^\infty \left |\mathbb E_{\mu_\omega}(\varphi_{\omega, 0} \cdot (\varphi_{\sigma^n \omega, 0}\circ T_\omega^n))\right | &\le C_\alpha\sum_{n=1}^\infty n^{-1/\alpha+1}\mathbb E_{\mathbb P}(\|\varphi_{\omega, 0}\|_{C^1}\cdot \|\varphi_{\sigma^n \omega, 0}\|_{C^1}) \\
			&\le C_\alpha\sum_{n=1}^\infty n^{-1/\alpha+1} (\mathbb E_{\mathbb P}(\|\varphi_{\omega, 0}\|_{C^1}^2))^{1/2} \cdot (\mathbb E_{\mathbb P}(\|\varphi_{\sigma^n \omega, 0}\|_{C^1}^2))^{1/2} \\
			&\le C_\alpha\mathbb E_{\mathbb P}(\|\varphi_\omega\|_{C^1}^2)\sum_{n=1}^\infty n^{-1/\alpha+1} ,
		\end{split}
		\]
		which is finite due to~\eqref{integ2} and $\alpha<1/2$.
		
		It remains to verify~\cite[Theorem 2.3 (ii')]{K}. More precisely, we need to show that
		\begin{equation}\label{integ3}
			\sum_{n=0}^\infty \mathbb E_{\mu_\omega}  |L_{\sigma^{-n}\omega}^n \varphi_{\sigma^{-n}\omega, 0}| \in L^2(\Omega, \mathcal F, \mathbb P),
		\end{equation}
		where 
		\[
		L_\omega(\psi):=\frac{\mathcal L_\omega (\psi h_\omega)}{h_{\sigma \omega}} \quad \text{and} \quad L_\omega^n:=L_{\sigma^{n-1}\omega} \circ \ldots \circ L_{\sigma \omega}\circ L_\omega,
		\]
		for $\omega \in \Omega$ and $n\in \mathbb N$. In addition, $L_\omega^0:=\Id$.  Note that 
		\[
		L_\omega^n (\psi)=\frac{\mathcal L_\omega^n(\psi h_\omega)}{h_{\sigma^n \omega}},
		\]
		for $\omega \in \Omega$ and $n\in \mathbb N$.
		Then, by relying on~\eqref{dec} again we see that
		\begin{equation}\label{eq:cond_2}
			\begin{split}
				\sum_{n=1}^\infty \mathbb E_{\mu_\omega}  |L_{\sigma^{-n}\omega}^n \varphi_{\sigma^{-n}\omega, 0}| &=\sum_{n=1}^\infty \|L_{\sigma^{-n}\omega}^n \varphi_{\sigma^{-n}\omega}\|_{L^1(\mu_\omega)} \\
				&=\sum_{n=1}^\infty \|\mathcal L_{\sigma^{-n}\omega}^n(\varphi_{\sigma^{-n}\omega, 0}h_{\sigma^{-n}\omega})\|_{L^1(m)} \\
				&\le C_\alpha\sum_{n=1}^\infty n^{-1/\alpha+1}\|\varphi_{\sigma^{-n}\omega, 0}h_{\sigma^{-n}\omega}\|_{L^1(m)} \\
				&\le C_\alpha\sum_{n=1}^\infty n^{-1/\alpha+1}\|\varphi_{\sigma^{-n}\omega, 0}\|_{C^1}.
			\end{split}
		\end{equation}
		Therefore, using the fact that $\sigma$ preserves $\mathbb P$ we have
		\[
		\left \|\sum_{n=1}^\infty \mathbb E_{\mu_\omega } |L_{\sigma^{-n}\omega}^n \varphi_{\sigma^{-n}\omega, 0}| \right \|_{L^2(\Omega, \mathcal F, \mathbb P)}\le C_\alpha \| \|\varphi_{\omega}\|_{C^1}\|_{L^2(\Omega, \mathcal F, \mathbb P)}\sum_{n=1}^\infty n^{-1/\alpha+1},
		\]
		which is finite due to~\eqref{integ2} and $\alpha <\frac 1 2$.  The proof of (i) is complete. \smallskip 
		
		The proof in the case of (ii) is similar to that of (i) and proceeds via an application of \cite[Theorem 2.3]{K} relying on \eqref{eq:fast_ml_2}. 
		To avoid repetition, we only give an outline of the proof. Note that $\varphi_{\omega, 0}=\varphi_\omega$.
		Assuming (ii), we obtain \cite[(2.7)]{K} using~\eqref{integ} since $(\mathbb E_{\mu_\omega}|\varphi_{\omega}|^2)^{1/2}\le \|\varphi_\omega\|_{C^1}$. Arguing as in \eqref{eq:cond_1}, using~\eqref{eq:fast_ml_2} and~\eqref{special_observable}, we obtain
		\[
		\begin{split}
			\sum_{n=1}^\infty \left |\mathbb E_{\mu_\omega}(\varphi_{\omega} \cdot (\varphi_{\sigma^n \omega}\circ T_\omega^n))\right | 
			&\le \sum_{n=1}^\infty \|\mathcal L_\omega^n (\varphi_{\omega} h_\omega)\|_{L^1(m)} \cdot \|\varphi_{\sigma^n \omega}\|_{C^1} \\
			&\le C \sum_{n=1}^\infty ( C_\omega + \Vert \varphi_\omega \Vert_{C^1}  ) n^{ 1 - (1 +  \min\{\alpha, \gamma \} )/\alpha}
			\cdot C_{\sigma^n \omega},
		\end{split}
		\]
		for $\bP$-a.e. $\omega \in \Omega$, where $C > 0$ 
		depends only on $a, \alpha, \gamma$. Consequently,
		\[
		\begin{split}
			&\mathbb E_{\mathbb P}\sum_{n=1}^\infty \left |\mathbb E_{\mu_\omega}(\varphi_{\omega} \cdot (\varphi_{\sigma^n \omega}\circ T_\omega^n))\right | \\
			&\le 
			C
			\sum_{n=1}^\infty 
			\mathbb E_{\mathbb P} [ ( C_\omega + \Vert \varphi_\omega \Vert_{C^1}  )   C_{\sigma^n \omega} ]
			 n^{ 1 - (1 +  \min\{\alpha, \gamma \} )/\alpha}  \\
			&\le 
			C\biggl\{  (  ( \mathbb E_{\mathbb P} ( C_{ \omega}^2 ))^{1/2}  
			+  ( \mathbb E_{\mathbb P} (  \Vert \varphi_\omega \Vert_{C^1}^2 ) )^{1/2}  ) ( \mathbb E_{\mathbb P} ( C_{ \omega}^2 ))^{1/2}  \biggr\} 
			\sum_{n=1}^\infty  n^{ 1 - (1 +  \min\{\alpha, \gamma \} )/\alpha} < \infty,
		\end{split}
		\]
		since $\alpha < (1 + \gamma) / 2$. Hence, \cite[Theorem 2.3 (i)]{K} holds.
		Finally, \cite[Theorem 2.3 (ii')]{K} follows by a computation similar to \eqref{eq:cond_2}, but using~\eqref{eq:fast_ml_2} instead of \eqref{dec} to obtain
		\[
		\begin{split}
			\sum_{n=1}^\infty \mathbb E_{\mu_\omega } |L_{\sigma^{-n}\omega}^n \varphi_{\sigma^{-n}\omega}| 
			&\le C  \sum_{n=1}^\infty
			( C_{ \sigma^{-n}  \omega }+ \|\varphi_{\sigma^{-n}\omega}\|_{C^1} ) 
			 n^{ 1 - (1 +  \min\{\alpha,\gamma\} )/\alpha}.
		\end{split}
		\]
	\end{proof}
	
	\begin{remark}\phantom{-}
		\begin{enumerate}
			\item 
			In the case where $(\Omega, \sigma)$ is a full shift, $\mathbb P$ is an ergodic invariant measure for $\sigma$, and $\varphi_\omega$ 
			are Lipschitz continuous with Lipschitz constants uniformly bounded in $\omega$,
			Theorem~\ref{qclt} (i)  was established in~\cite[Corollary 3.8]{Su2} (extending the previous results~\cite{NPT, NTV}).
			\item The proof of Theorem~\ref{qclt} yields a stronger conclusion, namely that the quenched law of iterated logarithm holds, provided that $\Sigma^2>0$ (see the statement of~\cite[Theorem 2.3.]{K}).
			\item To the best of our knowledge, Theorem~\ref{qclt} (ii) is new in the setting of random dynamics.
			For a single LSV map $T_\alpha$ with parameter 
			$1/2 < \alpha < 1$ and absolutely continuous invariant probability measure $\mu$, Gou\"{e}zel \cite{G}
			proved that $n^{-1/2} \sum_{j=0}^{n-1} \varphi \circ T_\alpha^n$ 
			converges weakly to a normal distribution on $([0,1], \mu)$, provided that $\varphi$ is H\"older continuous 
			with $\int \varphi \, d \mu = 0$ and 
			$|\varphi(x)| \le C x^\gamma$ holds for $\gamma > \alpha - 1/2$. A corresponding weak invariance principle was proved by Dedecker 
			and Merlev\`{e}de \cite{DM}.
			\item We present a simple class of examples of observables $\Phi$ satisfying~\eqref{special_observable}. To this end, take $C^1$ functions $u, g\colon [0, 1]\to \mathbb R$ such that 
			\[
			0\le u(x)\le Kx^\gamma \quad x\in [0, 1],
			\]
			where $\gamma>2\alpha-1$, and $u$ is not identically zero.
			Define now $\Phi \colon \Omega \times [0, 1]\to \mathbb R$ by
			\[
			\varphi_\omega=\Phi(\omega, \cdot):=u (g-c_\omega) \quad \omega \in \Omega,
			\]
			where 
			\[
			c_\omega:=\frac{\int_0^1 gu\, d\mu_\omega}{\int_0^1u\, d\mu_\omega}, \quad \omega \in \Omega.
			\]
			We note that $c_\omega$ is well-defined for $\mathbb P$-a.e. $\omega \in \Omega$. Indeed, we recall that it was established in the proof of~\cite[Proposition 9]{DGTS} that there is $\rho>0$ such that for $\mathbb P$-a.e. $\omega \in \Omega$, $h_\omega\ge \rho$ $m$-a.e. Hence, 
			\[
			\int_0^1 u\, d\mu_\omega \ge \rho \int_0^1u\, dm>0 \quad \text{for $\mathbb P$-a.e. $\omega \in \Omega$.}
			\]
			In addition, 
			\[
			|c_\omega|\le \frac{\left |\int_0^1 gu\, d\mu_\omega\right |}{\int_0^1 u\, d\mu_\omega}\le \frac{\|g\|_\infty \int_0^1 u\, d\mu_\omega}{\int_0^1 u\, d\mu_\omega}=\|g\|_\infty, \quad \text{for $\mathbb P$-a.e. $\omega \in \Omega$.}
			\]
			Therefore, 
			\[
			|\varphi_\omega(x)|=|u(x)| \cdot |g(x)-c_\omega|\le 2 Kx^\gamma \|g\|_\infty \quad \text{for $x\in [0, 1]$ and $\mathbb P$-a.e. $\omega \in \Omega$.}
			\]
			Thus, the first requirement in~\eqref{special_observable} holds with $C_\omega:=2K\|g\|_\infty$ for $\mathbb P$-a.e. $\omega \in \Omega$. Clearly, the second requirement in~\eqref{special_observable} is also fulfilled.
		\end{enumerate}
	\end{remark}

	\section{Continuity of the variance}\label{sec:varcont}
	We are interested in studying the regularity of the variance $\Sigma^2$ in Theorem~\ref{qclt} under suitable perturbations of random dynamics.
	
	Let us start by introducing the appropriate setup. We take $\varepsilon_0>0$ and measurable maps $\beta \colon \Omega \to (0, 1)$ and $\delta \colon \Omega \to (0, 1)$ such that 
	\[
	\alpha:=\esssup_{\omega}\beta (\omega)+\varepsilon_0 <\frac 1 2 \quad \text{and} \quad \essinf_{\omega}\beta(\omega)-\varepsilon_0>0.
	\]
	For $\omega \in \Omega$ and $\varepsilon \in (-\varepsilon_0, \varepsilon_0)$ let $T_{\omega, \varepsilon}$ be an LSV map with parameter $\beta (\omega)+\varepsilon \delta (\omega)$. We will write $T_{\omega}$ instead of $T_{\omega, 0}$. By $\mathcal L_{\omega, \varepsilon}$ we will denote the transfer operator associated with $T_{\omega, \varepsilon}$, again writing $\mathcal L_\omega$ instead of $\mathcal L_{\omega, 0}$.
	By $(\mu_{\omega, \varepsilon})_{\omega \in \Omega}$ we will denote the random a.c.i.m associated with the cocycle $(T_{\omega, \varepsilon})_{\omega \in \Omega}$. We will write $\mu_\omega$ instead of $\mu_{\omega, 0}$. By $h_{\omega, \varepsilon}\in \mathcal C_*\cap \mathcal C_2$, we denote the density of $\mu_{\omega, \varepsilon}$ (again writing $h_\omega$ instead of $h_{\omega, 0}$). 
	
	Let $F\colon \Omega \times [0, 1] \to \mathbb R$ be a measurable map such that $F(\omega, \cdot)\in C^1[0, 1]$ for $\omega \in \Omega$ and 
	\begin{equation}\label{observ}
		\esssup_{\omega \in \Omega}\|F(\omega, \cdot)\|_{C^1}<+\infty.
	\end{equation}
	In the sequel, we will denote $F_\omega:=F(\omega, \cdot)$. Set
	\[
	f_{\omega, \varepsilon}:=F_\omega-\int_0^1F_\omega \, d_{\mu_{\omega, \varepsilon}}, \quad \omega \in \Omega, \ \varepsilon \in (-\varepsilon_0, \varepsilon_0).
	\]
	Moreover, let 
	\begin{equation}\label{Sigmavar}
		\Sigma_\varepsilon^2:=\int_\Omega \int_0^1f_{\omega, \varepsilon}^2 \, d_{\mu_{\omega, \varepsilon}} \, d\mathbb P(\omega)+2\sum_{n=1}^\infty \int_\Omega \int_0^1 f_{\omega, \varepsilon}\cdot (f_{\sigma^n \omega, \varepsilon}\circ T_{\omega, \varepsilon}^n)\, d\mu_{\omega, \varepsilon}\, d\mathbb P(\omega).
	\end{equation}
	By Theorem~\ref{qclt}, $\Sigma_\varepsilon^2$ is finite for $\varepsilon \in (-\varepsilon_0, \varepsilon_0)$.
	Then, we have the following theorem.
	\begin{theorem}\label{contvariance}
		The map $\varepsilon \mapsto \Sigma_\varepsilon^2$ is continuous in $0$. 
	\end{theorem}
	\begin{proof}
		We will first show that
		\begin{equation}\label{claim1}
			\lim_{\varepsilon \to 0}\int_\Omega\int_0^1 f_{\omega, \varepsilon}^2\, d \mu_{\omega,  \varepsilon}\, d\mathbb P(\omega)=\int_\Omega\int_0^1 f_{\omega, 0}^2\, d\mu_\omega\, d\mathbb P(\omega).
		\end{equation}
		To this end, observe that 
		\[
		\begin{split}
			\int_\Omega\int_0^1 f_{\omega, \varepsilon}^2\, d\mu_{\omega,  \varepsilon}\, d\mathbb P(\omega)-\int_\Omega\int_0^1 f_{\omega, 0}^2\, d\mu_\omega\, d\mathbb P(\omega) &=\int_\Omega \int_0^1 (f_{\omega, \varepsilon}^2-f_{\omega, 0}^2)\, d\mu_{\omega, \varepsilon}\, d\mathbb P(\omega)\\
			&\phantom{=}+\int_\Omega \int_0^1 f_{\omega, 0}^2(h_{\omega, \varepsilon}-h_{\omega})\,dm\, d\mathbb P(\omega).
		\end{split}
		\]
		Observe that for $\mathbb P$-a.e. $\omega \in \Omega$,
		\begin{equation}\label{difference1}
			f_{\omega, \varepsilon}-f_{\omega, 0}=\int_0^1F_\omega (h_\omega-h_{\omega, \varepsilon})\, dm,
		\end{equation}
		and consequently
		\[
		|f_{\omega, \varepsilon}-f_{\omega, 0}|\le \|F_\omega\|_{C^0}\|h_{\omega, \varepsilon}-h_{\omega, 0}\|_{L^1(m)} \le C_{\alpha, F} |\varepsilon|,
		\]
		where in the second inequality we used Proposition~\ref{statstab} and~\eqref{observ}. On the other hand,
		\[
		\|f_{\omega, \varepsilon}+f_{\omega, 0}\|_{C^0}\le 4\|F_\omega \|_{C^0}.
		\]
		Consequently,
		\[
		\|f_{\omega, \varepsilon}^2-f_{\omega, 0}^2\|_{C^0}\le C_{\alpha, F}|\varepsilon|,
		\]
		for $\mathbb P$-a.e. $\omega \in \Omega$. This implies that 
		\[
		\begin{split}
			\left |\int_\Omega \int_0^1 (f_{\omega, \varepsilon}^2-f_{\omega, 0}^2)\, d\mu_{\omega, \varepsilon}\, d\mathbb P(\omega)\right | &\le \int_\Omega \|f_{\omega, \varepsilon}^2-f_{\omega, 0}^2\|_{C^0}\, d\mathbb P(\omega) \le C_{\alpha, F}|\varepsilon|, 
		\end{split}
		\]
		yielding 
		\[
		\lim_{\varepsilon \to 0}\int_\Omega \int_0^1 (f_{\omega, \varepsilon}^2-f_{\omega, 0}^2)\, d\mu_{\omega, \varepsilon}\, d\mathbb P(\omega)=0.
		\]
		Similarly, $|f_{\omega, 0}^2|\le 4\|F_\omega \|_{C^0}^2$ and thus using Proposition~\ref{statstab} and~\eqref{observ} we have
		\[
		\lim_{\varepsilon \to 0}\int_\Omega \int_0^1 f_{\omega, 0}^2(h_{\omega, \varepsilon}-h_{\omega})\,dm\, d\mathbb P(\omega)=0,
		\]
		since 
		\[
		\begin{split}
			\left |\int_\Omega \int_0^1 f_{\omega, 0}^2(h_{\omega, \varepsilon}-h_{\omega})\,dm\, d\mathbb P(\omega)\right | \le \int_\Omega \|f_{\omega, 0}^2\|_{C^0}\cdot \|h_{\omega, \varepsilon}-h_\omega \|_{L^1(m)}\, d\mathbb P(\omega)\le C_{\alpha, F} |\varepsilon|.
		\end{split}
		\]
		Therefore, we conclude that~\eqref{claim1} holds.

		Next, we observe  that 
		\[
		\begin{split}
			\left |\int_\Omega \int_0^1 f_{\omega, \varepsilon}\cdot (f_{\sigma^n \omega, \varepsilon}\circ T_{\omega, \varepsilon}^n)\, d\mu_{\omega, \varepsilon}\, d\mathbb P(\omega) \right |  &=\left | \int_\Omega\int_0^1 \mathcal L_{\omega, \varepsilon}^n(f_{\omega, \varepsilon}h_{\omega, \varepsilon})f_{\sigma^n \omega, \varepsilon}\, dm\, d\mathbb P(\omega) \right | \\
			&\le \int_\Omega \|\mathcal L_{\omega, \varepsilon}^n(f_{\omega, \varepsilon}h_{\omega, \varepsilon})\|_{L^1(m)}\cdot \|f_{\sigma^n \omega, \varepsilon}\|_{C^0}\, d\mathbb P(\omega) \\
			&\le C_\alpha n^{-1/\alpha+1}\left (\esssup\|F_\omega \|_{C^1} \right )^2,
		\end{split}
		\]
		where we used~\eqref{dec} and~\eqref{observ}. Since the series $\sum_{n=1}^\infty n^{-1/\alpha+1}$ is summable, in order to prove our theorem, it remains to show that for each $n\in \mathbb N$, the map 
		\[
		\varepsilon \mapsto  \int_\Omega \int_0^1 f_{\omega, \varepsilon}\cdot (f_{\sigma^n \omega, \varepsilon}\circ T_{\omega, \varepsilon}^n)\, d\mu_{\omega, \varepsilon}\, d\mathbb P(\omega)=\int_\Omega\int_0^1 \mathcal L_{\omega, \varepsilon}^n(f_{\omega, \varepsilon}h_{\omega, \varepsilon})f_{\sigma^n \omega, \varepsilon}\, dm\, d\mathbb P(\omega)
		\]
		is continuous in $0$. To this end, we first observe that 
		\[
		\begin{split}
			&\int_\Omega\int_0^1 \mathcal L_{\omega, \varepsilon}^n(f_{\omega, \varepsilon}h_{\omega, \varepsilon})f_{\sigma^n \omega, \varepsilon}\, dm\, d\mathbb P(\omega)-\int_\Omega\int_0^1 \mathcal L_{\omega}^n(f_{\omega, 0}h_{\omega})f_{\sigma^n \omega, 0}\, dm\, d\mathbb P(\omega)\\
			&=\int_\Omega \int_0^1 \mathcal L_\omega^n(f_{\omega, 0}h_\omega)(f_{\sigma^n \omega, \varepsilon}-f_{\sigma^n \omega, 0})\, dm\, d\mathbb P(\omega)+\int_\Omega \int_0^1 (\mathcal L_{\omega, \varepsilon}^n-\mathcal L_\omega^n)(f_{\omega, 0}h_\omega) f_{\sigma^n \omega, \varepsilon}\, dm\, d\mathbb P(\omega)\\
			&\phantom{=}+\int_\Omega \int_0^1 \mathcal L_{\omega, \varepsilon}^n(f_{\omega, \varepsilon}h_{\omega, \varepsilon}-f_{\omega, 0}h_\omega) f_{\sigma^n \omega, \varepsilon}\, dm\, d\mathbb P(\omega) \\
			&=\int_\Omega \int_0^1 (\mathcal L_{\omega, \varepsilon}^n-\mathcal L_\omega^n)(f_{\omega, 0}h_\omega) f_{\sigma^n \omega, \varepsilon}\, dm\, d\mathbb P(\omega)+\int_\Omega \int_0^1 \mathcal L_{\omega, \varepsilon}^n(f_{\omega, \varepsilon}h_{\omega, \varepsilon}-f_{\omega, 0}h_\omega) f_{\sigma^n \omega, \varepsilon}\, dm\, d\mathbb P(\omega),
		\end{split}
		\]
		since $f_{\sigma^n \omega, \varepsilon}-f_{\sigma^n \omega, 0}$ depends only on $\varepsilon$ and $\omega $ (see~\eqref{difference1}) and 
		\[
		\int_0^1 \mathcal L_{\omega}^n (f_{\omega, 0}h_\omega)\, dm=\int_0^1 f_{\omega, 0}h_\omega \, dm=0.
		\]
		Next, 
		\[
		\begin{split}
			& \left |\int_\Omega \int_0^1 \mathcal L_{\omega, \varepsilon}^n(f_{\omega, \varepsilon}h_{\omega, \varepsilon}-f_{\omega, 0}h_\omega) f_{\sigma^n \omega, \varepsilon}\, dm\, d\mathbb P(\omega) \right | \\
			&\le \int_\Omega \|\mathcal L_{\omega, \varepsilon}^n(f_{\omega, \varepsilon}h_{\omega, \varepsilon}-f_{\omega, 0}h_\omega)\|_{L^1(m)}\cdot \|f_{\sigma^n \omega, \varepsilon}\|_{C^0}\, d\mathbb P(\omega) \\
			&\le \int_\Omega \|f_{\omega, \varepsilon}h_{\omega, \varepsilon}-f_{\omega, 0}h_\omega\|_{L^1(m)}\cdot \|f_{\sigma^n \omega, \varepsilon}\|_{C^0}\, d\mathbb P(\omega) \\
			&\le 2\esssup_{\omega \in \Omega}\|F_\omega\|_{C^0}\int_\Omega \|f_{\omega, \varepsilon}h_{\omega, \varepsilon}-f_{\omega, 0}h_\omega\|_{L^1(m)}\, d\mathbb P(\omega) .
		\end{split}
		\]
		Moreover, 
		\[
		\begin{split}
			\|f_{\omega, \varepsilon}h_{\omega, \varepsilon}-f_{\omega, 0}h_\omega\|_{L^1(m)} &\le \|f_{\omega, \varepsilon}\|_{C^0}\|h_{\omega, \varepsilon}-h_\omega\|_{L^1(m)}+|f_{\omega, \varepsilon}-f_{\omega, 0}| \\
			&\le C_{\alpha, F}|\varepsilon| ,
		\end{split}
		\]
		where we used Proposition~\ref{statstab}. This together with~\eqref{observ} implies that 
		\[
		\lim_{\varepsilon \to 0}\int_\Omega \int_0^1 \mathcal L_{\omega, \varepsilon}^n(f_{\omega, \varepsilon}h_{\omega, \varepsilon}-f_{\omega, 0}h_\omega) f_{\sigma^n \omega, \varepsilon}\, dm\, d\mathbb P(\omega)=0.
		\]
		Finally, we have that 
		\[
		\begin{split}
			&\int_\Omega \int_0^1 (\mathcal L_{\omega, \varepsilon}^n-\mathcal L_\omega^n)(f_{\omega, 0}h_\omega) f_{\sigma^n \omega, \varepsilon}\, dm\, d\mathbb P(\omega) \\
			&=\sum_{j=1}^n \int_\Omega \int_0^1 \mathcal L_{\sigma^j \omega, \varepsilon}^{n-j}(\mathcal L_{\sigma^{j-1}\omega, \varepsilon}-\mathcal L_{\sigma^{j-1}\omega}) \mathcal L_\omega^{j-1}(f_{\omega, 0}h_\omega)f_{\sigma^n \omega, \varepsilon}\, dm\, d\mathbb P(\omega).
		\end{split}
		\]
		Therefore, 
		\[
		\begin{split}
			&\left |\int_\Omega \int_0^1 (\mathcal L_{\omega, \varepsilon}^n-\mathcal L_\omega^n)(f_{\omega, 0}h_\omega) f_{\sigma^n \omega, \varepsilon}\, dm\, d\mathbb P(\omega) \right | \\
			&\le \sum_{j=1}^n \int_\Omega \|\mathcal L_{\sigma^j \omega, \varepsilon}^{n-j}(\mathcal L_{\sigma^{j-1}\omega, \varepsilon}-\mathcal L_{\sigma^{j-1}\omega}) \mathcal L_\omega^{j-1}(f_{\omega, 0}h_\omega)\|_{L^1(m)}\cdot \|f_{\sigma^n \omega, \varepsilon}\|_{C^0}\, d\mathbb P(\omega) \\
			&\le \sum_{j=1}^n \int_\Omega \|(\mathcal L_{\sigma^{j-1}\omega, \varepsilon}-\mathcal L_{\sigma^{j-1}\omega}) \mathcal L_\omega^{j-1}(f_{\omega, 0}h_\omega)\|_{L^1(m)}\cdot \|f_{\sigma^n \omega, \varepsilon}\|_{C^0}\, d\mathbb P(\omega) \\
			&\le 2 \left (\esssup_{\omega \in \Omega}\|F_\omega\|_{C^0}\right )\sum_{j=1}^n \int_\Omega \|(\mathcal L_{\sigma^{j-1}\omega, \varepsilon}-\mathcal L_{\sigma^{j-1}\omega}) \mathcal L_\omega^{j-1}(f_{\omega, 0}h_\omega)\|_{L^1(m)}\, d\mathbb P(\omega).
		\end{split}
		\]
		In addition,  we have
		\[
		(\mathcal L_{\sigma^{j-1}\omega, \varepsilon}-\mathcal L_{\sigma^{j-1}\omega}) \mathcal L_\omega^{j-1}(f_{\omega, 0}h_\omega)=-\int_{\beta(\sigma^{j-1}\omega)}^{\beta(\sigma^{j-1}\omega)+\varepsilon \delta (\sigma^{j-1}\omega)} (X_\gamma N_\gamma (\mathcal L_\omega^{j-1}(f_{\omega, 0}h_\omega)))'\, d\gamma,
		\]
		since $\partial_\gamma \mathcal L_\gamma(\phi)=-(X_\gamma N_\gamma (\phi))'$ (see~\cite[p.865]{BT}). It follows from $h_\omega \in \mathcal C_*\cap \mathcal C_2$ and
		 \eqref{eq:decomp_c1} that, for $\mathbb P$-a-e. $\omega \in \Omega$ we have $f_{\omega, 0}h_\omega=\psi_1-\psi_2$ for some $\psi_i=\psi_i(\omega)\in \mathcal C_*\cap C^1(0,1]$.  Moreover, $\|\psi_i\|_{L^1}\le K$ for $i=1, 2$, where $K$ does not depend on $\omega$.
		Hence, 
		\[
		\mathcal L_\omega^{j-1}(f_{\omega, 0}h_\omega)=\mathcal L_\omega^{j-1}(\psi_1)-\mathcal L_\omega^{j-1}(\psi_2), \quad \mathcal L_\omega^{j-1}(\psi_i)\in \mathcal C_*\cap C^1(0,1], \ i=1, 2.
		\]
		Consequently
		\begin{align*}
			\Vert 
			(X_\gamma N_\gamma (\mathcal L_\omega^{j-1}(f_{\omega, 0}h_\omega)))'
			\Vert_{L^1(m)}
			\le \sum_{i=1,2} \Vert 
			(X_\gamma N_\gamma (\mathcal L_\omega^{j-1}( \psi_i )))'
			\Vert_{L^1(m)} \le \tilde{K},
		\end{align*}
		where the last inequality is a consequence of bounds on $X_\gamma$ and $X_\gamma'$ in~\cite[(2.3) and (2.4)]{BT}, the invariance of $\cC_* \cap C^1(0,1]$ with respect to $N_\gamma$
		and $\cL_\gamma$ (see \cite[Proposition 5]{DGTS}), and the 
		observation (see~\cite[Remark 1]{DGTS}) that
		\begin{align}\label{eq:cone_fun_diff}
			\phi \in \cC_* \cap C^1(0,1] \implies 
			|\phi'(x)| \le (\alpha + 1) x^{-1} \phi(x) \quad 
			\forall x \in (0,1].
		\end{align}
		This gives that 
		\[
		\int_\Omega \|(\mathcal L_{\sigma^{j-1}\omega, \varepsilon}-\mathcal L_{\sigma^{j-1}\omega}) \mathcal L_\omega^{j-1}(f_{\omega, 0}h_\omega)\|_{L^1(m)}\, d\mathbb P(\omega) \le C|\varepsilon|,
		\]
		for some $C>0$ independent on $\varepsilon$. Thus, 
		\[
		\lim_{\varepsilon \to 0}\int_\Omega \int_0^1 (\mathcal L_{\omega, \varepsilon}^n-\mathcal L_\omega^n)(f_{\omega, 0}h_\omega) f_{\sigma^n \omega, \varepsilon}\, dm\, d\mathbb P(\omega)=0.
		\]
		This completes the proof of the theorem.
	\end{proof}
	
	\section{Differentiability of the variance}

	Take $\beta \colon \Omega\to (0, 1)$, $\delta \colon \Omega \to [0, 1)$ be measurable maps and $\underline{\alpha}>0$ such that $0<\underline{\alpha}<\alpha<\frac 1 2$,
	\[
	0<\underline{\alpha}\le \essinf_{\omega \in \Omega}\beta(\omega)-\varepsilon_0 \quad \text{and} \quad \esssup_{\omega \in \Omega}\beta(\omega)+\varepsilon_0\le \alpha.
	\]
	Let $T_{\omega, \varepsilon}$ be the LSV map with parameter $\beta(\omega)+\varepsilon \delta (\omega)$ for every $\omega \in \Omega$ and $\varepsilon \in (-\varepsilon_0, \varepsilon_0)$. By $(\mu_{\omega, \varepsilon})_{\omega \in \Omega}$ we will denote the random a.c.i.m that corresponds to the cocycle $(T_{\omega, \varepsilon})_{\omega \in \Omega}$, where $h_{\omega, \varepsilon}$ is the density of $\mu_{\omega, \varepsilon}$. As in the previous section,  we will drop the subscript $\varepsilon$ when $\varepsilon=0$. 
	
	We will use the following result, which is a sharper version of~\cite[Theorem 14]{DGTS}. Its proof is given in Appendix C.
	\begin{theorem}\label{sharpl}
		There exists $C=C_{\alpha, \underline{\alpha}}>0$ such that for $\varepsilon$ sufficiently close to $0$,
		\begin{equation}\label{stronger}
			\|\varepsilon^{-1}(h_{\omega, \varepsilon}-h_\omega)-\hat h_\omega \|_{L^1(m)}\le C|\varepsilon|^{1-2\alpha},
		\end{equation}
		for $\mathbb P$-a.e. $\omega \in \Omega$,
		where
		\begin{equation}\label{series}
			\hat h_\omega:=-\sum_{i=0}^\infty \delta (\sigma^{-(i+1)}\omega) \mathcal L_{\sigma^{-i}\omega}^i \left(X_{\beta(\sigma^{-(i+1)}\omega)} N_{\beta(\sigma^{-(i+1)}\omega)} (h_{\sigma^{-(i+1)}\omega}) \right)' \in L^1(m).
		\end{equation}
		Moreover, 
		\begin{equation}\label{series2}
			\esssup_{\omega \in \Omega}\|\hat h_\omega \|_{L^1(m)}<+\infty.
		\end{equation}
	\end{theorem}
	
	\begin{remark}\label{sharplc}
		In Theorem~\ref{sharpl} we show for $\mathbb P$-a.e. $\omega \in \Omega$, the map $ (-\varepsilon_0, \varepsilon_0)\ni \varepsilon \mapsto h_{\omega, \varepsilon}\in L^1(m)$ is differentiable in $\varepsilon=0$. On the other hand, in~\cite[Theorem 14]{DGTS} a weaker statement is proved: for $\mathbb P$-a.e. $\omega \in \Omega$ and $\psi \in L^\infty(m)$, the map $\varepsilon \mapsto \int_0^1\psi h_{\omega, \varepsilon}\, dm$ is differentiable in $\varepsilon=0$. For other quenched linear response results for different classes of random dynamics, we refer to~\cite{CN,DGS, DH2, DS, SR}.
	\end{remark}

    \begin{remark}
An anonymous referee raised a question about the possibility of obtaining higher order derivatives of the map $\varepsilon \mapsto h_{\omega, \varepsilon}$ in $\varepsilon=0$. We emphasize that this problem has not been addressed in the case of autonomous dynamics yet. However, it is briefly mentioned in~\cite[p.859]{BT} that the cone approach developed in~\cite{BT} (used in~\cite{DGTS} and the present paper) is likely to yield such higher order regularity. However, it would require the construction of new families of invariant cones. In addition, we mention that if such result holds, then it would have a somewhat different form than~\cite[Theorem 3.6]{CN}. More precisely, unlike in~\cite{CN}, where the quadratic response holds in the space weaker than that in which we have a linear response, in our setting we should have both in $L^1(m)$.
\end{remark}

	Let $F\colon \Omega \times [0, 1] \to \mathbb R$ be a measurable map such that $F(\omega, \cdot)\in C^2[0, 1]$ for $\omega \in \Omega$ and that
	\begin{equation}\label{observ2}
		\esssup_{\omega \in \Omega}\|F(\omega, \cdot)\|_{C^2}<+\infty.
	\end{equation}
Let $f_{\omega, \varepsilon}$ and $\Sigma_\varepsilon^2$ be as in the previous section. Then we have the following theorem.
\begin{theorem}\label{diffvar}
	Let $\alpha <\frac 1 5$. Then, 
	the map $\varepsilon \mapsto \Sigma_\varepsilon^2$ is differentiable in $0$.
\end{theorem}

\begin{remark}\label{formula}
	We note that the proof of Theorem~\ref{diffvar} yields  
	\[
	\begin{split}
		\frac{d}{d\ve}\Sigma_\ve^2 \big\rvert_{\ve=0} &=\int_\Omega \int_0^1 f_{\omega, 0}^2 \hat h_\omega \, dm\, d\mathbb P(\omega)-2\int_\Omega \left (\int_0^1 F_\omega \hat h_\omega \, dm\right )\int_0^1 f_{\omega, 0}\, d\mu_\omega \, d\mathbb P(\omega) \\
		&\phantom{=}+\sum_{n=1}^\infty \int_\Omega \int_0^1 \mathcal L_\omega^n (-L(\omega, F_\omega)h_\omega+f_{\omega, 0}\hat h_\omega)f_{\sigma^n \omega, 0}\, dm\, d\mathbb P(\omega) \\
		&\phantom{=}+\sum_{n=1}^\infty \sum_{j=0}^{n-1}\int_\Omega \delta (\sigma^j \omega)\int_0^1\mathcal L_{\sigma^{j+1}\omega}^{n-j-1} \left (\partial_\gamma \mathcal L_{\beta(\sigma^j \omega)}(\mathcal L_\omega^j(f_{\omega, 0}h_\omega))\right )f_{\sigma^n \omega, 0}\, dm\, d\mathbb P(\omega),
	\end{split}
	\]
	where $L(\omega, F_\omega):=\int_0^1 F_\omega \hat h_\omega \, dm$. 
	The proof of Theorem~\ref{diffvar} shows that this formula for 
	the derivative of 
	$\Sigma_\ve^2$ is well-defined 
	under the assumption $\alpha < 1/5$. In particular, Lemma~\ref{summability2} ensures the summability
	\begin{align}\label{eq:diff_last_term}
		\sum_{n=1}^\infty \sum_{j=0}^{n-1}\int_\Omega \int_0^1  |\mathcal L_{\sigma^{j+1}\omega}^{n-j-1} \left (\partial_\gamma \mathcal L_{\beta(\sigma^j \omega)}(\mathcal L_\omega^j(f_{\omega, 0}h_\omega))\right ) | \, dm\, d\mathbb P(\omega) < \infty.
	\end{align}
	We do not know whether \eqref{eq:diff_last_term} remains valid beyond 
	the range $\alpha < 1/5$.
\end{remark}

\begin{remark}
	For a special class of observables as in 
	Theorem \ref{qclt} (ii), we can can weaken the constraint on 
	$\alpha$ to $\alpha < 1/4 + \delta$ for suitable $\delta > 0$.
	See Theorem \ref{thm:special} for a precise statement.
\end{remark}

The proof of Theorem~\ref{diffvar} will be established in a series of auxiliary results.  
\subsection{Differentiability of the first term in~\eqref{Sigmavar}}
\begin{lemma}
	The map
	\begin{equation}\label{6:05}
		\varepsilon \to \int_\Omega \int_0^1f_{\omega, \varepsilon}^2 \, d\mu_{\omega, \varepsilon}\, d\mathbb P(\omega)
	\end{equation}
	is differentiable in $\varepsilon=0$.
\end{lemma}

\begin{proof}
	Writing 
	\[
	\begin{split}
		&\varepsilon^{-1}\left (\int_\Omega \int_0^1f_{\omega, \varepsilon}^2 \, d\mu_{\omega, \varepsilon}\, d\mathbb P(\omega)-\int_\Omega \int_0^1f_{\omega, 0}^2 \, d\mu_{\omega}\, d\mathbb P(\omega)\right ) \\
		&=\varepsilon^{-1}\int_\Omega\int_0^1 (f_{\omega, \varepsilon}^2-f_{\omega, 0}^2)\, d\mu_{\omega, \varepsilon}\, d\mathbb P(\omega)+\varepsilon^{-1}\int_\Omega\int_0^1f_{\omega, 0}^2(h_{\omega, \varepsilon}-h_{\omega})\, dm\, d\mathbb P(\omega),
	\end{split}
	\]
	it follows from~\eqref{stronger} that 
	\[
	\lim_{\varepsilon \to 0}\frac{1}{\varepsilon}\int_\Omega \int_0^1f_{\omega, 0}^2(h_{\omega, \varepsilon}-h_{\omega, 0})\, dm\, d\mathbb P(\omega)=\int_\Omega \int_0^1 f_{\omega, 0}^2\hat h_\omega\, dm\, d\mathbb P(\omega), 
	\]
	since 
	\[
	\begin{split}
		&\left |\int_\Omega \int_0^1f_{\omega, 0}^2(\varepsilon^{-1}(h_{\omega, \varepsilon}-h_{\omega, 0}))\, dm\, d\mathbb P(\omega)-\int_\Omega \int_0^1 f_{\omega, 0}^2\hat h_\omega\, dm\, d\mathbb P(\omega)\right | \\
		&\le \int_\Omega \|f_{\omega, 0}^2\|_{C^0}\cdot \|\varepsilon^{-1}(h_{\omega, \varepsilon}-h_{\omega, 0})-\hat h_\omega\|_{L^1(m)}\, d\mathbb P(\omega) \le C_{\alpha, \underline{\alpha}, F}|\varepsilon|^{1-2\alpha},
	\end{split}
	\]
	for $\varepsilon$ sufficiently close to $0$.
	Next, we note that 
	\[
	\lim_{\varepsilon\to 0}\frac{1}{\varepsilon}\left (\int_\Omega\int_0^1 (f_{\omega, \varepsilon}^2-f_{\omega, 0}^2)\, d\mu_{\omega, \varepsilon}\, d\mathbb P(\omega)-\int_\Omega\int_0^1 (f_{\omega, \varepsilon}^2-f_{\omega, 0}^2)\, d\mu_{\omega}\, d\mathbb P(\omega)\right )=0,
	\]
	as
	\[
	\begin{split}
		&\left |\int_\Omega\int_0^1 (f_{\omega, \varepsilon}^2-f_{\omega, 0}^2)\, d\mu_{\omega, \varepsilon}\, d\mathbb P(\omega)-\int_\Omega\int_0^1 (f_{\omega, \varepsilon}^2-f_{\omega, 0}^2)\, d\mu_{\omega}\, d\mathbb P(\omega)\right | \\
		&\le \int_\Omega \|f_{\omega, \varepsilon}^2-f_{\omega, 0}^2\|_{C^0}\|h_{\omega, \varepsilon}-h_\omega \|_{L^1(m)}\, d\mathbb P(\omega) \\
		&\le C_{\alpha, F}|\varepsilon|^2,
	\end{split}
	\]
	for some $C_{\alpha, F}>0$ where we used~\eqref{difference1} and Proposition~\ref{statstab}.
	Finally, 
	\[
	\lim_{\varepsilon \to 0}\frac{1}{\varepsilon}\int_\Omega\int_0^1 (f_{\omega, \varepsilon}^2-f_{\omega, 0}^2)\, d\mu_{\omega}\, d\mathbb P(\omega)=-2\int_\Omega \left (\int_0^1F_\omega \hat h_\omega\, dm\right )\int_0^1 f_{\omega, 0}\, d\mu_\omega\, d\mathbb P(\omega),
	\]
	since 
	\[
	\begin{split}
		&\left \|\varepsilon^{-1}(f_{\omega, \varepsilon}^2-f_{\omega, 0}^2)+2\left  (\int_0^1F_\omega \hat h_\omega\, dm\right )f_{\omega, 0} \right \|_{C^0} \\
		&\le \left |\varepsilon^{-1}(f_{\omega, \varepsilon}-f_{\omega, 0})+\int_0^1 F_\omega \hat h_\omega\, dm\right | \cdot \|f_{\omega, \varepsilon}+f_{\omega, 0}\|_{C^0}+\left |\int_0^1 F_\omega \hat h_\omega \, dm\right | \cdot |f_{\omega, \varepsilon}-f_{\omega, 0}| \\
		&\le C_{\alpha, \underline{\alpha}, F}|\varepsilon|^{1-2\alpha}+C_{\alpha, F} |\varepsilon|,
	\end{split}
	\]
	for $\varepsilon$ sufficiently close to $0$ where we used~\eqref{difference1}, Proposition~\ref{statstab} and Theorem~\ref{sharpl}.
	Thus, the map in~\eqref{6:05} is differentiable in $\varepsilon=0$.
\end{proof}

\subsection{Differentiability of the second term in~\eqref{Sigmavar}}
Next, we consider the second term in~\eqref{Sigmavar}:
\[
\sum_{n=1}^\infty \int_\Omega \int_0^1 (f_{\sigma^n \omega, \varepsilon}\circ T_{\omega, \varepsilon}^n)f_{\omega, \varepsilon}\, d\mu_{\omega, \varepsilon}\, d\mathbb P(\omega)=\sum_{n=1}^\infty \int_\Omega \int_0^1 \mathcal L_{\omega, \varepsilon}^n (f_{\omega, \varepsilon}h_{\omega, \varepsilon})f_{\sigma^n \omega, \varepsilon}\, dm\, d\mathbb P(\omega).
\]
Let 
\[
C_n(\varepsilon):=\int_\Omega \int_0^1 \mathcal L_{\omega, \varepsilon}^n (f_{\omega, \varepsilon}h_{\omega, \varepsilon})f_{\sigma^n \omega, \varepsilon}\, dm\, d\mathbb P(\omega),
\]
and
\[
D_n(\varepsilon):=\frac{C_n(\varepsilon)-C_n(0)}{\varepsilon}.
\]
Observe that 
\[
|C_n(\varepsilon)| \le \int_\Omega \|\mathcal L_{\omega, \varepsilon}^n(f_{\omega, \varepsilon}h_{\omega, \varepsilon} )\|_{L^1(m)} \cdot \|f_{\sigma^n \omega, \varepsilon}\|_{C^0}\, d\mathbb P(\omega) \le C_{\alpha, F} n^{-1/\alpha+1}.
\]
Consequently, 
\[
|D_n(\varepsilon)| \le C_{\alpha, F} |\varepsilon|^{-1}n^{-1/\alpha+1}.
\]
Let us now take $\gamma \in (0, \frac 1 2)$ such that 
\begin{equation}\label{gamma1}
	\gamma (1/\alpha-2)>1.
\end{equation}
We note that this is possible since $\alpha <\frac 1 4$.
Then, we have
\[
\left |\sum_{n\ge |\varepsilon|^{-\gamma}}D_n(\varepsilon)\right |\le C|\varepsilon|^{\gamma (1/\alpha-2)-1},
\]
where $C>0$ is not dependent on $\varepsilon$.
Hence, due to~\eqref{gamma1} we have
\begin{equation}\label{w00}
	\lim_{\varepsilon \to 0}\sum_{n\ge |\varepsilon|^{-\gamma}}D_n(\varepsilon)=0.
\end{equation}

Next,  note that 
\begin{equation}\label{w22}
	\begin{split}
		D_n(\varepsilon)
		&=\int_\Omega \int_0^1 \mathcal L_\omega^n(f_{\omega, 0}h_\omega)(\varepsilon^{-1}(f_{\sigma^n \omega, \varepsilon}-f_{\sigma^n \omega, 0}))\, dm\, d\mathbb P(\omega)\\
		&\phantom{=}+\frac{1}{\varepsilon}\int_\Omega \int_0^1 (\mathcal L_{\omega, \varepsilon}^n-\mathcal L_\omega^n)(f_{\omega, 0}h_\omega) f_{\sigma^n \omega, \varepsilon}\, dm\, d\mathbb P(\omega)\\
		&\phantom{=}+\int_\Omega \int_0^1 \mathcal L_{\omega, \varepsilon}^n(\varepsilon^{-1}(f_{\omega, \varepsilon}h_{\omega, \varepsilon}-f_{\omega, 0}h_\omega)) f_{\sigma^n \omega, \varepsilon}\, dm\, d\mathbb P(\omega)\\
		&=\frac{1}{\varepsilon}\int_\Omega \int_0^1 (\mathcal L_{\omega, \varepsilon}^n-\mathcal L_\omega^n)(f_{\omega, 0}h_\omega) f_{\sigma^n \omega, \varepsilon}\, dm\, d\mathbb P(\omega)\\
		&\phantom{=}+\int_\Omega \int_0^1 \mathcal L_{\omega, \varepsilon}^n(\varepsilon^{-1}(f_{\omega, \varepsilon}h_{\omega, \varepsilon}-f_{\omega, 0}h_\omega)) f_{\sigma^n \omega, \varepsilon}\, dm\, d\mathbb P(\omega)\\
		&=:D_n^1(\varepsilon)+D_n^2(\varepsilon),
	\end{split}
\end{equation}
since as in the proof of Theorem~\ref{contvariance} we have 
\[
\int_\Omega \int_0^1 \mathcal L_\omega^n(f_{\omega, 0}h_\omega)(\varepsilon^{-1}(f_{\sigma^n \omega, \varepsilon}-f_{\sigma^n \omega, 0}))\, dm\, d\mathbb P(\omega)=0.
\]
We first focus on $D_n^2(\varepsilon)$. Taking into account~\eqref{difference1} and
\[
\int_0^1 \mathcal L_{\omega, \varepsilon}^n (f_{\omega, \varepsilon}h_{\omega, \varepsilon}-f_{\omega, 0}h_\omega)\, dm=\int_0^1 (f_{\omega, \varepsilon}h_{\omega, \varepsilon}-f_{\omega, 0}h_\omega)\, dm=0,
\]
we have
\[
D_n^2(\varepsilon)=\int_\Omega \int_0^1 \mathcal L_{\omega, \varepsilon}^n(\varepsilon^{-1}(f_{\omega, \varepsilon}h_{\omega, \varepsilon}-f_{\omega, 0}h_\omega)) f_{\sigma^n \omega, 0}\, dm\, d\mathbb P(\omega).
\]
Let
\[
\bar D_n^2(\varepsilon):=\int_\Omega \int_0^1\mathcal L_{\omega, \varepsilon}^n(-L(\omega, F_\omega)h_\omega+f_{\omega,0}\hat h_\omega)f_{\sigma^n \omega, 0}\, dm\, d\mathbb P(\omega),
\]
where $L(\omega, F_\omega):=\int_0^1 F_\omega \hat h_\omega\, dm$.
\begin{lemma}\label{837n}
	We have
	\[
	|D_n^2(\varepsilon)-\bar D_n^2(\varepsilon)|\le C_{\alpha, \underline{\alpha}, F}|\varepsilon|^{1-2\alpha},
	\]
	for $\varepsilon$ sufficiently close to $0$.
\end{lemma}
\begin{proof}
	Note that 
	\[
	\begin{split}
		&|D_n^2(\varepsilon)-\bar D_n^2(\varepsilon)|\\
		&\le \int_\Omega \|\varepsilon^{-1}(f_{\omega, \varepsilon}-f_{\omega, 0})h_{\omega, \varepsilon}+L(\omega, F_\omega)h_\omega+f_{\omega, 0}(\varepsilon^{-1}(h_{\omega, \varepsilon}-h_\omega)-\hat h_\omega)\|_{L^1(m)}\cdot \|f_{\sigma^n \omega, 0}\|_{C^0}\, d\mathbb P(\omega)\\
		&\le C_F\int_\Omega \|\varepsilon^{-1}(f_{\omega, \varepsilon}-f_{\omega, 0})h_{\omega, \varepsilon}+L(\omega, F_\omega)h_\omega+f_{\omega, 0}(\varepsilon^{-1}(h_{\omega, \varepsilon}-h_\omega)-\hat h_\omega)\|_{L^1(m)}\, d\mathbb P(\omega).
	\end{split}
	\]
	Moreover, 
	\[
	\begin{split}
		\|\varepsilon^{-1}(f_{\omega, \varepsilon}-f_{\omega, 0})h_{\omega, \varepsilon}+L(\omega, F_\omega)h_\omega\|_{L^1(m)}  &\le \|\varepsilon^{-1}(f_{\omega, \varepsilon}-f_{\omega, 0})h_{\omega, \varepsilon}+L(\omega, F_\omega)h_{\omega, \varepsilon}\|_{L^1(m)} \\
		&\phantom{\le}+|L(\omega, F_\omega)|\|h_{\omega, \varepsilon}-h_\omega\|_{L^1(m)}\\
		&= \left |\varepsilon^{-1}(f_{\omega, \varepsilon}-f_{\omega, 0})+L(\omega, F_\omega)\right |\\
		&\phantom{=}+|L(\omega, F_\omega)|\|h_{\omega, \varepsilon}-h_\omega\|_{L^1(m)}\\
		&\le C_{\alpha, \underline{\alpha}, F}|\varepsilon|^{1- 2\alpha}+C_{\alpha, F} |\varepsilon| \\
		&\le C_{\alpha, \underline{\alpha}, F}|\varepsilon|^{1-2\alpha},
	\end{split}
	\]
	for $\mathbb P$-a.e. $\omega \in \Omega$ and $\varepsilon$ sufficiently close to $0$, where we used Proposition~\ref{statstab} and Theorem~\ref{sharpl}.
	Similarly,\[
	\begin{split}
		\|f_{\omega, 0}(\varepsilon^{-1}(h_{\omega, \varepsilon}-h_\omega)-\hat h_\omega)\|_{L^1(m)}&\le \|f_{\omega, 0}\|_{C^0} \cdot \|\varepsilon^{-1}(h_{\omega, \varepsilon}-h_\omega)-\hat h_\omega\|_{L^1(m)}\\
		&\le C_{\alpha, \underline{\alpha}, F}|\varepsilon|^{1-2\alpha},
	\end{split}
	\]
	for $\mathbb P$-a.e. $\omega \in \Omega$ and $\varepsilon$ sufficiently close to $0$. The conclusion of the lemma follows readily from the above estimates.
\end{proof}
By Lemma~\ref{837n}, we have \[
\sum_{n<|\varepsilon|^{-\gamma}}| D_n^2(\varepsilon)-\bar D_n^2(\varepsilon)| \le C_{\alpha, \underline{\alpha}, F}|\varepsilon|^{1-2\alpha-\gamma}. 
\]
Since $
1-2\alpha>\gamma$,
we have 
\begin{equation}\label{w11}
	\lim_{\varepsilon \to 0}\sum_{n<|\varepsilon|^{-\gamma}}| D_n^2(\varepsilon)-\bar D_n^2(\varepsilon)|=0.
\end{equation}

Taking into account~\eqref{w00}, \eqref{w22} and~\eqref{w11}, in order to conclude the differentiability of $\varepsilon \mapsto \Sigma^2_\ve$, it remains to establish the following:
\begin{itemize}
	\item The limits $\lim_{\ve \to 0} D_n^1(\ve)$ and $\lim_{\ve \to 0}\bar D_n^2(\varepsilon)$ exist for each $n \ge 1$.
	\item There exist $\bar{\ve}_0 > 0$ and $u_n$ such that 
	for all $n < |\ve|^{-\gamma}$ and all $\ve \in [-\bar{\ve}_0, \bar{\ve}_0]\setminus \{0\}$,
	$$
	|D_n^1(\ve)| \le u_n, \quad | \bar D_n^2(\ve)|\le u_n \quad \text{and} \quad 
	\sum_{n=1}^\infty u_n < \infty.
	$$
\end{itemize}
The first point above is established in Lemma~\ref{limitexists1} for $\bar D_n^2$ and in Lemma~\ref{limitexists2} for $D_n^1$. The second point is proved in Lemma~\ref{summability1} for $\bar D_n^2$ and in Lemma~\ref{summability2} for $D_n^1$.

\subsection{A series of auxiliary lemmas}
We write
\begin{equation}\label{8:41}
	\begin{split}
		\bar D_n^2(\varepsilon) &=\int_\Omega \int_0^1\mathcal L_{\omega, \varepsilon}^n(-L(\omega, F_\omega)h_\omega+f_{\omega,0}\hat h_\omega)f_{\sigma^n \omega, 0}\, dm\, d\mathbb P(\omega)\\
		&=\int_\Omega \int_0^1 \mathcal L_{\omega, \varepsilon}^n (f_{\omega, 0}\hat h_\omega)f_{\sigma^n \omega, 0}\, dm\, d\mathbb P(\omega)-\int_\Omega L(\omega, F_\omega) \int_0^1 \mathcal L_{\omega, \varepsilon}^n (h_\omega)f_{\sigma^n \omega, 0}\, dm\, d\mathbb P(\omega)\\
		&=: \bar D_{n, 1}^2(\varepsilon)-\bar D_{n, 2}^2(\varepsilon).
	\end{split}
\end{equation}
\begin{lemma}\label{lem:aux1}
	We have
	\[
	|\bar D_{n, 2}^2(\varepsilon)|\le C_{\alpha, F}(n^{-1/\gamma}+n^{-1/\alpha+1}), \quad n<|\varepsilon|^{-\gamma}.
	\]
\end{lemma}
\begin{proof}
	Note that 
	\[
	\|\mathcal L_{\omega, \varepsilon}^n h_\omega -h_{\sigma^n \omega, \varepsilon}\|_{L^1(m)}=\|\mathcal L_{\omega, \varepsilon}^n (h_\omega-h_{\omega, \varepsilon})\|_{L^1(m)}\le C_\alpha n^{-1/\alpha+1}.
	\]
	Then, using that 
	\[
	\int_0^1 \mathcal L_{\omega, \varepsilon}^n (h_{\omega, \varepsilon})f_{\sigma^n \omega, \varepsilon}\, dm=\int_0^1 h_{\sigma^n \omega, \varepsilon}f_{\sigma^n \omega, \varepsilon}\, dm=0,
	\]
	we have 
	\[
	\begin{split}
		&\left |\int_0^1 \mathcal L_{\omega, \varepsilon}^n(h_\omega)f_{\sigma^n \omega, 0}\, dm\right | \\
		&\le \left |\int_0^1 \mathcal L_{\omega, \varepsilon}^n(h_\omega)(f_{\sigma^n \omega, \varepsilon}-f_{\sigma^n \omega, 0})\, dm\right |+\left |\int_0^1 \mathcal L_{\omega, \varepsilon}^n (h_\omega)f_{\sigma^n \omega, \varepsilon}\, dm\right | \\
		&=\left |\int_0^1 \mathcal L_{\omega, \varepsilon}^n(h_\omega)(f_{\sigma^n \omega, \varepsilon}-f_{\sigma^n \omega, 0})\, dm\right |+\left |\int_0^1 \mathcal L_{\omega, \varepsilon}^n (h_\omega-h_{\omega, \varepsilon})f_{\sigma^n \omega, \varepsilon}\, dm\right | \\
		&\le \|\mathcal L_{\omega, \varepsilon}^n h_\omega \|_{L^1(m)}\cdot \left |f_{\sigma^n \omega, \varepsilon}-f_{\sigma^n \omega, 0}\right |+\|\mathcal L_{\omega, \varepsilon}^n(h_\omega-h_{\omega, \varepsilon})\|_{L^1(m)}\cdot \|f_{\sigma^n \omega, \varepsilon}\|_{C^0} \\
		&\le  \left |f_{\sigma^n \omega, \varepsilon}-f_{\sigma^n \omega, 0}\right |+\|\mathcal L_{\omega, \varepsilon}^n(h_\omega-h_{\omega, \varepsilon})\|_{L^1(m)}\cdot \|f_{\sigma^n \omega, \varepsilon}\|_{C^0} \\
		&\le C_{\alpha, F}|\varepsilon|+C_{\alpha, F}n^{-1/\alpha+1},
	\end{split}
	\]
	which yields the desired conclusion.
\end{proof}
\begin{lemma}\label{lem:aux2}
	We have
	\[
	|\bar D_{n, 1}^2(\varepsilon)| \le C_{\alpha, \underline{\alpha}, F}( 
	n^{1-\frac{1}{2\alpha}}
	+n^{-1/\gamma}), \quad n<|\varepsilon|^{-\gamma}.
	\]
\end{lemma}
\begin{proof}
	By using Theorem~\ref{sharpl} we have
	\begin{equation}\label{gh1}
		\begin{split}
			\left |\int_0^1 \mathcal L_{\omega, \varepsilon}^n (f_{\omega, 0}\hat h_\omega)f_{\sigma^n \omega, 0}\, dm\right | &\le \left |\int_0^1 \mathcal L_{\omega, \varepsilon}^n (f_{\omega, 0}(\hat h_\omega-r^{-1}(h_{\omega, r}-h_\omega)))
			f_{\sigma^n \omega, 0}\, dm\right | \\
			&\phantom{\le}+\left |\int_0^1 \mathcal L_{\omega, \varepsilon}^n (f_{\omega, 0}r^{-1}(h_{\omega, r}-h_\omega))f_{\sigma^n \omega, 0}\, dm\right | \\
			&\le \|\mathcal L_{\omega, \varepsilon}^n (f_{\omega, 0}(\hat h_\omega-r^{-1}(h_{\omega, r}-h_\omega)))\|_{L^1(m)} \cdot \|f_{\sigma^n \omega, 0}\|_{C^0}\\
			&\phantom{\le}+\left |\int_0^1 \mathcal L_{\omega, \varepsilon}^n (f_{\omega, 0}r^{-1}(h_{\omega, r}-h_\omega))f_{\sigma^n \omega, 0}\, dm\right |\\
			&\le C_{\alpha, \underline{\alpha}, F}r^{1-2\alpha}+ \left |\int_0^1 \mathcal L_{\omega, \varepsilon}^n (f_{\omega, 0}r^{-1}(h_{\omega, r}-h_\omega))f_{\sigma^n \omega, 0}\, dm\right |,
		\end{split}
	\end{equation}
	for $r\in (0, 1)$ that will be fixed later on.
	Next,
	\[
	\begin{split}
		& \int_0^1 \mathcal L_{\omega, \varepsilon}^n (f_{\omega, 0}r^{-1}(h_{\omega, r}-h_\omega))f_{\sigma^n \omega, 0}\, dm \\
		&=\int_0^1 \mathcal L_{\omega, \varepsilon}^n \left (r^{-1}\left (f_{\omega, 0}-\int_0^1f_{\omega, 0}h_{\omega, r}\, dm\right )h_{\omega, r}\right )f_{\sigma^n \omega, 0}\, dm \\
		&\phantom{=}+\int_0^1 \mathcal L_{\omega, \varepsilon}^n \left (r^{-1}\left (\int_0^1 f_{\omega, 0}h_{\omega, r}\, dm\right )h_{\omega, r}\right )f_{\sigma^n \omega, 0}\, dm \\
		&\phantom{=}-\int_0^1 \mathcal L_{\omega, \varepsilon}^n \left(r^{-1} \left (f_{\omega, 0}-\int_0^1 f_{\omega, 0}h_\omega\, dm\right )h_\omega \right )f_{\sigma^n \omega, 0}\, dm \\
		&\phantom{=}-\int_0^1 \mathcal L_{\omega, \varepsilon}^n \left (r^{-1} \left (\int_0^1 f_{\omega, 0}h_\omega\, dm\right )h_\omega \right )f_{\sigma^n \omega, 0}\, dm.
	\end{split}
	\]
	Since 
	\[
	\int_0^1 \left (f_{\omega, 0}-\int_0^1f_{\omega, 0}h_{\omega, r}\, dm\right )h_{\omega, r}\, dm=0,
	\]
	we have using~\eqref{dec} that 
	\begin{equation}\label{gh2}
		\begin{split}
			&\left | \int_0^1 \mathcal L_{\omega, \varepsilon}^n \left (r^{-1}\left (f_{\omega, 0}-\int_0^1f_{\omega, 0}h_{\omega, r}\, dm\right )h_{\omega, r}\right )f_{\sigma^n \omega, 0}\, dm\right | \\
			&\le C_F r^{-1} \left \|\mathcal L_{\omega, \varepsilon}^n \left (\left (f_{\omega, 0}-\int_0^1f_{\omega, 0}h_{\omega, r}\, dm\right )h_{\omega, r} \right )\right \|_{L^1(m)} \\
			&\le C_{\alpha, F} r^{-1}n^{-1/\alpha+1} \left \|f_{\omega, 0}-\int_0^1f_{\omega, 0}h_{\omega, r}\, dm \right \|_{C^1} \\
			&\le C_{\alpha, F} r^{-1}n^{-1/\alpha+1}.
		\end{split}
	\end{equation}
	In the same manner, 
	\begin{equation}\label{gh3}
		\left |\int_0^1 \mathcal L_{\omega, \varepsilon}^n \left(r^{-1} \left (f_{\omega, 0}-\int_0^1 f_{\omega, 0}h_\omega\, dm\right )h_\omega \right )f_{\sigma^n \omega, 0}\, dm\right |\le C_{\alpha, F} r^{-1}n^{-1/\alpha+1}.
	\end{equation}
	Moreover, 
	\[
	\begin{split}
		&\int_0^1 \mathcal L_{\omega, \varepsilon}^n \left (r^{-1}\left (\int_0^1 f_{\omega, 0}h_{\omega, r}\, dm\right )h_{\omega, r}\right )f_{\sigma^n \omega, 0}\, dm\\
		&-\int_0^1 \mathcal L_{\omega, \varepsilon}^n \left (r^{-1} \left (\int_0^1 f_{\omega, 0}h_\omega\, dm\right )h_\omega \right )f_{\sigma^n \omega, 0}\, dm \\
		&=\int_0^1 \mathcal L_{\omega, \varepsilon}^n \left (r^{-1} \left(\int_0^1f_{\omega, 0}h_{\omega, r}\, dm\right ) (h_{\omega, r}-h_{\omega}) \right )f_{\sigma^n \omega, 0}\, dm \\
		&\phantom{=}+\int_0^1 \mathcal L_{\omega, \varepsilon}^n \left (r^{-1}\left (\int_0^1 f_{\omega, 0}h_{\omega, r}\, dm-\int_0^1f_{\omega, 0}h_\omega \, dm\right )h_\omega \right )f_{\sigma^n \omega, 0}\, dm.
	\end{split}
	\]
	Observe that 
	\begin{equation}\label{gh4}
		\begin{split}
			& \left |\int_0^1 \mathcal L_{\omega, \varepsilon}^n \left (r^{-1} \left (\int_0^1f_{\omega, 0}h_{\omega, r}\, dm\right ) (h_{\omega, r}-h_{\omega}) \right )f_{\sigma^n \omega, 0}\, dm \right |\\
			&\le C_F r^{-1}\left |\int_0^1 f_{\omega, 0}h_{\omega, r}\, dm\right | \cdot \|\mathcal L_{\omega, \varepsilon}^n (h_{\omega, r}-h_\omega)\|_{L^1(m)} \\
			&\le C_{\alpha, F}  r^{-1}n^{-1/\alpha+1}.
		\end{split}
	\end{equation}
	In addition, 
	\[
	\begin{split}
		&\int_0^1 \mathcal L_{\omega, \varepsilon}^n \left (r^{-1}\left (\int_0^1 f_{\omega, 0}h_{\omega, r}\, dm-\int_0^1f_{\omega, 0}h_\omega \, dm\right )h_\omega \right )f_{\sigma^n \omega, 0}\, dm \\
		&=\int_0^1 \mathcal L_{\omega, \varepsilon}^n \left ( \left (r^{-1}\left (\int_0^1 f_{\omega, 0}h_{\omega, r}\, dm-\int_0^1f_{\omega, 0}h_\omega \, dm\right )-\int_0^1f_{\omega, 0}\hat h_\omega \, dm\right )h_\omega \right )f_{\sigma^n \omega, 0}\, dm \\
		&\phantom{=}+\int_0^1 f_{\omega, 0}\hat h_\omega \, dm \int_0^1 \mathcal L_{\omega, \varepsilon}^n (h_\omega) f_{\sigma^n \omega, 0}\, dm.
	\end{split}
	\]
	By Theorem~\ref{sharpl},\begin{equation}\label{gh5}
		\begin{split}
			& \left |\int_0^1 \mathcal L_{\omega, \varepsilon}^n \left ( \left (r^{-1}\left (\int_0^1 f_{\omega, 0}h_{\omega, r}\, dm-\int_0^1f_{\omega, 0}h_\omega \, dm\right )-\int_0^1f_{\omega, 0}\hat h_\omega \, dm\right )h_\omega \right )f_{\sigma^n \omega, 0}\, dm\right | \\
			&\le C_F \left |r^{-1}\left (\int_0^1 f_{\omega, 0}h_{\omega, r}\, dm-\int_0^1f_{\omega, 0}h_\omega \, dm\right )-\int_0^1f_{\omega, 0}\hat h_\omega \, dm\right |  \\
			&\le C_{\alpha, \underline{\alpha}, F} r^{1-2\alpha}.
		\end{split}
	\end{equation}
	Furthermore, 
	\begin{equation}\label{gh6}
		\begin{split}
			& \left |\int_0^1 f_{\omega, 0}\hat h_\omega \, dm \int_0^1 \mathcal L_{\omega, \varepsilon}^n (h_\omega) f_{\sigma^n \omega, 0}\, dm-\int_0^1 f_{\omega, 0}\hat h_\omega \, dm \int_0^1 h_{\sigma^n \omega, \varepsilon} f_{\sigma^n \omega, 0}\, dm \right | \\
			&\le C_F \|\mathcal L_{\omega, \varepsilon}^n (h_\omega-h_{\omega, \varepsilon})\|_{L^1(m)} \\
			&\le C_{\alpha, F} n^{-1/\alpha+1}, 
		\end{split}
	\end{equation} and thus (taking into account \eqref{gh1}-\eqref{gh6})
	\[
	\left |\bar D_{n, 1}^2(\varepsilon)-\int_\Omega \int_0^1 f_{\omega, 0}\hat h_\omega \, dm \int_0^1 h_{\sigma^n \omega, \varepsilon} f_{\sigma^n \omega, 0}\, dm \, d\mathbb P(\omega)\right | \le C_{\alpha, \underline{\alpha}, F} (r^{1-2\alpha}+r^{-1}n^{-1/\alpha+1}),
	\]
	for $r$ close to $0$. In addition, since $\int_0^1 f_{\omega, \varepsilon}h_{\omega, \varepsilon}\, dm=0$,
	\[
	\begin{split}
		& \int_\Omega \int_0^1 f_{\omega, 0}\hat h_\omega \, dm \int_0^1 h_{\sigma^n \omega, \varepsilon} f_{\sigma^n \omega, 0}\, dm \, d\mathbb P(\omega)\\
		&=\int_\Omega \int_0^1 f_{\omega, 0}\hat h_\omega \, dm \int_0^1 h_{\sigma^n \omega, \varepsilon} (f_{\sigma^n \omega, 0}-f_{\sigma^n \omega, \varepsilon})\, dm \, d\mathbb P(\omega),
	\end{split}
	\]
	which gives that
	\[
	\left |\int_\Omega \int_0^1 f_{\omega, 0}\hat h_\omega \, dm \int_0^1 h_{\sigma^n \omega, \varepsilon} f_{\sigma^n \omega, 0}\, dm \, d\mathbb P(\omega)\right | \le C_{\alpha, 
		F}|\varepsilon| \le C_{\alpha, F}n^{-1/\gamma},
	\]
	for $n<|\varepsilon|^{-\gamma}$. By choosing $r=n^{-\frac{1-\alpha}{\alpha (2-2\alpha)}}$ (so that $r^{1-2\alpha}=r^{-1}n^{-1/\alpha+1}$), we obtain the desired conclusion. 
\end{proof}
As a direct consequence of the previous two lemmas and~\eqref{8:41}, we obtain the following.
\begin{lemma}\label{summability1}
	We have
	\[
	|\bar{D}_n^2(\varepsilon)| \le C_{\alpha, \underline{\alpha}, F} (n^{1-\frac{1}{2\alpha}}+n^{-1/\gamma}+n^{-1/\alpha+1}), \quad n<|\varepsilon|^{-\gamma}.
	\]
\end{lemma}

\begin{lemma}\label{limitexists1}
	For each $n\in \mathbb N$,
	\[
	\lim_{\varepsilon \to 0}\bar D_n^2(\varepsilon)=\bar D_n^2(0).
	\]
\end{lemma}

\begin{proof}
	We first show that \begin{equation}\label{limit1}\lim_{\ve \to 0} \bar{D}_{n,1}^2(\ve) = \bar{D}_{n,1}^2(0).\end{equation} We start by decomposing
	\begin{align*}
		\hat{h}_\omega &= - \sum_{i=0}^\infty 
		\delta( \sigma^{ - ( i + 1 ) } \omega ) 
		\cL_{ \sigma^{-i} \omega  }^i   ( X_{ \beta( \sigma^{ - ( i + 1) }  \omega ) }  
		N_{\beta( \sigma^{ - ( i + 1) } \omega )}( h_{ \sigma^{-(i+1)} \omega } ))' \\
		&= - \sum_{i=0}^{k} 
		\delta( \sigma^{ - ( i + 1 ) } \omega ) 
		\cL_{ \sigma^{-i} \omega  }^i   ( X_{ \beta( \sigma^{ - ( i + 1) }  \omega ) }  
		N_{\beta( \sigma^{ - ( i + 1) } \omega )}( h_{ \sigma^{-(i+1)} \omega } ))' \\
		&\quad - \sum_{i= k + 1 }^\infty 
		\delta( \sigma^{ - ( i + 1 ) } \omega ) 
		\cL_{ \sigma^{-i} \omega  }^i   ( X_{ \beta( \sigma^{ - ( i + 1) }  \omega ) }  
		N_{\beta( \sigma^{ - ( i + 1) } \omega )}( h_{ \sigma^{-(i+1)} \omega } ))'\\
		&=: \hat{h}_{\omega,k} + R_{\omega, k},
	\end{align*}
	where $k\in \mathbb N$. It follows from Theorem~\ref{DECC} and Lemma~\ref{LEM12} that for $\mathbb P$-a.e. $\omega \in \Omega$, 
	\[
	\|R_{\omega, k}\|_{L^1(m)} \le C_\alpha \sum_{i = k + 1}^\infty i^{ - 1 / \alpha + 1 } 
	\le C_{\alpha} k^{ - 1 / \alpha + 2 }\quad \text{for $k\in \mathbb N$.}
	\]
	Thus, for $k\in \mathbb N$,
	\begin{align*}
		|\bar{D}_{n,1}^2(\ve) - \bar{D}_{n,1}^2(0)|
		&\le \biggl| 
		\int_\Omega \int_0^1 
		[ \cL_{\omega, \ve}^n ( f_{\omega, 0}   \hat{h}_{\omega, k} ) 
		- \cL_{\omega, 0}^n ( f_{\omega, 0}   \hat{h}_{\omega, k} ) ] 
		f_{\sigma^n\omega, 0} \, d m \, d \bP(\omega)
		\biggr| \\
		&\phantom{\le}+ C_{\alpha, F} k^{ - 1 / \alpha + 2}.
	\end{align*}
	By $h_\omega \in \cC_* \cap \cC_2$
	and
	Lemma~\ref{LEM12},
	$$
	( X_{ \beta( \sigma^{ - ( i + 1) }  \omega ) } N_{\beta( \sigma^{ - ( i + 1) } \omega )}( h_{ \sigma^{-(i+1)} \omega } ))' 
	= \psi_1 - \psi_2,
	$$
	where $\psi_1, \psi_2 \in \cC_* \cap C^1(0,1]$. Hence, \cite[Lemma 3.3]{LS} implies that
	$$
	f_{\omega, 0}
	\cL_{\sigma^{-i} \omega }^i ( X_{ \beta( \sigma^{ - ( i + 1) }  \omega ) } N_{\beta( \sigma^{ - ( i + 1) } \omega )}( h_{ \sigma^{-(i+1)} \omega } ))'
	= \tilde{\psi}_{i,1} - \tilde{\psi}_{i,2},
	$$
	where $\tilde{\psi}_{i,j} \in C^1(0,1] \cap \cC_*$. Consequently,
	\begin{align}\label{eq:decom}
		f_{\omega, 0} \hat{h}_{\omega, k} = - \sum_{i=0}^k \sum_{j=1,2} 
		\delta( \sigma^{-(i+1)} \omega ) (-1)^{j - 1 } \tilde{\psi}_{i,j}.
	\end{align}
	Substituting \eqref{eq:decom} into the remaining expression, and using 
	$\varphi(x + \ve) - \varphi(x) = \ve \int_{0}^1 \varphi'( x +  t \ve) \, dt$,
	we obtain
	\begin{align*}
		&\int_\Omega \int_0^1 
		[ \cL_{\omega, \ve}^n ( f_{\omega, 0}   \hat{h}_{\omega, k} ) 
		- \cL_{\omega, 0}^n ( f_{\omega, 0}   \hat{h}_{\omega, k} ) ] 
		f_{\sigma^n\omega, 0} \, d m \, d \bP(\omega) \\
		&= - \sum_{i=0}^k \sum_{j=1,2}
		\delta( \sigma^{ - ( i + 1 ) } \omega ) (-1)^{j - 1 }
		\int_\Omega \int_0^1 
		[ \cL_{\omega, \ve}^n (   \tilde{\psi}_{i,j}  ) 
		- \cL_{\omega, 0}^n (    \tilde{\psi}_{i,j}   ) ] 
		f_{\sigma^n\omega, 0} \, d m \, d \bP(\omega) \\
		&= - \sum_{i=0}^k \sum_{j=1,2} \sum_{\ell = 0}^{ n - 1}
		\delta( \sigma^{ - ( i + 1 ) } \omega ) (-1)^{j - 1 }
		\int_\Omega \int_0^1 
		\cL_{ \sigma^{n -  \ell } \omega, \ve }^\ell [ \cL_{ \beta( \sigma^{n-1-\ell} \omega  ) + \ve \delta(\sigma^{n-1-\ell}  \omega ) } - \cL_{ \beta( \sigma^{n-1-\ell} \omega  ) }   ]   \\
		&\quad \times \cL_{ \omega }^{n-1-\ell} ( \tilde{\psi}_{i,j} )
		f_{\sigma^n\omega, 0} \, d m \, d \bP(\omega) \\
		&= - \sum_{i=0}^k \sum_{j=1,2} \sum_{\ell = 0}^{ n - 1}
		\delta( \sigma^{ - ( i + 1 ) } \omega ) (-1)^{j - 1 }
		\ve \delta(\sigma^{n-1-\ell}  \omega ) \\
		&\quad \times \int_\Omega \int_0^1 \int_0^1
		\cL_{ \sigma^{n -  \ell } \omega, \ve }^\ell \partial_\gamma \cL_{ \beta( \sigma^{n-1-\ell} \omega  ) + t \ve \delta(\sigma^{n-1-\ell}  \omega ) }   \cL_{ \omega }^{n-1-\ell} ( \tilde{\psi}_{i,j} ) 
		f_{\sigma^n\omega, 0} \,  dt \, d m \, d \bP(\omega).
	\end{align*}
	Moreover,
	\begin{align*}
		&\partial_\gamma \cL_{ \beta( \sigma^{n-1-\ell} \omega  ) + t \ve \delta(\sigma^{n-1-\ell}  \omega ) }   \cL_{ \omega }^{n-1-\ell} ( 
		\tilde{\psi}_{i,j} ) \\
		&= ( X_{  \beta( \sigma^{n-1-\ell} \omega  ) + t \ve \delta(\sigma^{n-1-\ell}  \omega ) } N_{  \beta( \sigma^{n-1-\ell} \omega  ) + t \ve \delta(\sigma^{n-1-\ell}  \omega ) } (   \cL_{ \omega }^{n-1-\ell} ( 
		\tilde{\psi}_{i,j} )  ) )'.
	\end{align*}
	It follows from \eqref{eq:cone_fun_diff},
	bounds on $X_\gamma$ and $X_\gamma'$ in~\cite[(2.3) and (2.4)]{BT}, 
	and invariance properties of the cones in \eqref{cone_inv} that
	$$
	\Vert ( X_{  \beta( \sigma^{n-1-\ell} \omega  ) + t \ve \delta(\sigma^{n-1-\ell}  \omega ) } N_{  \beta( \sigma^{n-1-\ell} \omega  ) + t \ve \delta(\sigma^{n-1-\ell}  \omega ) } (   \cL_{ \omega }^{n-1-\ell} ( \tilde{\psi}_{i,j} )  ) )' \Vert_{L^1(m)} \le C_{\alpha, \underline{\alpha}, F}.
	$$ 
	Consequently,
	\begin{align*}
		&\biggl| \int_\Omega \int_0^1 
		[ \cL_{\omega, \ve}^n ( f_{\omega, 0}   \hat{h}_{\omega, k} ) 
		- \cL_{\omega, 0}^n ( f_{\omega, 0}   \hat{h}_{\omega, k} ) ] 
		f_{\sigma^n\omega, 0} \, d m \, d \bP(\omega)  \biggr| 
		\le C_{\alpha, \underline{\alpha}, F} |\ve| kn.
	\end{align*}
	We choose $k = \lfloor  |\ve|^{ - \alpha / ( 1 - \alpha )  } \rfloor$ so that 
	$|\ve| k$ and $k^{- 1 / \alpha + 2}$ are approximately equal. Since $\alpha < 1/2$, 
	$$
	|\bar{D}_{n,1}^2(\ve) - \bar{D}_{n,1}^2(0)| \le C_{\alpha, \underline{\alpha}, F} (|\ve| kn +  k^{- 1 / \alpha + 2} )\le C_{\alpha, \underline{\alpha}, F, n}|\ve|^{ 1 - \alpha / ( 1 - \alpha ) } \to 0 \quad 
	\text{as $\ve \to 0$}.
	$$
	Hence, \eqref{limit1} holds.
	
	We now claim that 
	\begin{equation}\label{limit2}
		\lim_{\ve \to 0}\bar{D}_{n, 2}^2(\ve)=\bar{D}_{n, 2}(0).
	\end{equation}
	Note that 
	\[
	\begin{split}
		\mathcal L_{\omega, \varepsilon}^n (h_\omega)-\mathcal L_\omega^n (h_\omega) &=\sum_{j=0}^{n-1} \mathcal L_{\sigma^{j+1} \omega, \varepsilon}^{n-j-1}(\mathcal L_{\sigma^{j}\omega, \varepsilon}-\mathcal L_{\sigma^{j}\omega})\mathcal L_\omega^{j}(h_\omega) \\
		&=\sum_{j=0}^{n-1} \int_{\beta(\sigma^{j}\omega)}^{\beta(\sigma^{j}\omega)+\ve \delta(\sigma^{j}\omega)}\mathcal L_{\sigma^{j +1}\omega, \varepsilon}^{n-j-1}[\partial_\gamma \mathcal L_\gamma (h_{\sigma^{j}\omega})]\, d\gamma \\
		&=-\sum_{j=0}^{n-1} \int_{\beta(\sigma^{j}\omega)}^{\beta(\sigma^{j}\omega)+\ve \delta(\sigma^{j}\omega)} \mathcal L_{\sigma^{j+1} \omega, \varepsilon}^{n-j-1}[(X_\gamma N_\gamma(h_{\sigma^{j}\omega}))']\, d\gamma.
	\end{split}
	\]
	By applying Lemma~\ref{LEM12} we conclude (using the convention $0^{-1/\alpha+1}:=1$) that 
	\[
	\|\mathcal L_{\omega, \varepsilon}^n (h_\omega)-\mathcal L_\omega^n (h_\omega)\|_{L^1(m)}\le C_\alpha |\ve|\sum_{j=0}^{n-1} (n-j-1)^{-1/\alpha+1}\le C_\alpha|\ve|,
	\]
	which easily implies~\eqref{limit2}. Finally, it remains to note that  the conclusion of the lemma follows readily from~\eqref{limit1} and~\eqref{limit2}.
\end{proof}

We now consider $D_n^1(\varepsilon)$ and note that 
\[
D_n^1(\varepsilon)=\frac{1}{\varepsilon}\int_\Omega \int_0^1 (\mathcal L_{\omega, \varepsilon}^n-\mathcal L_\omega^n)(f_{\omega, 0}h_\omega) f_{\sigma^n \omega, 0}\, dm\, d\mathbb P(\omega).
\]
\begin{lemma}\label{summability2}
	We have
	\begin{equation}\label{eq:summability2-1}
		|D_n^1(\varepsilon)| \le C_{\alpha, \underline{\alpha}, F}(n^{ 1 - 1 / \gamma }
		+
		n^{1-\frac{1-\alpha}{2\alpha}}), \quad n<|\varepsilon|^{-\gamma}.
	\end{equation}
\end{lemma}
\begin{proof}
	Observe that \begin{equation}\label{D12}
		\begin{split}
			&D_n^1(\varepsilon)= \\
			&=\int_\Omega \int_0^1 \sum_{j=0}^{n-1}\mathcal L_{\sigma^{j+1}\omega, \varepsilon}^{n-j-1}\left (\varepsilon^{-1}(\mathcal L_{\sigma^j \omega, \varepsilon}-\mathcal L_{\sigma^j \omega})\mathcal L_{\omega}^j(f_{\omega, 0}h_\omega)\right)f_{\sigma^n \omega, 0}\, dm\, d\mathbb P(\omega) \\
			&=\int_\Omega\int_0^1 \sum_{j=0}^{n-1}\delta (\sigma^j \omega)\mathcal L_{\sigma^{j+1}\omega, \varepsilon}^{n-j-1}\left (\partial_\gamma \mathcal L_{\beta(\sigma^j \omega)}(\mathcal L_\omega^j(f_{\omega, 0}h_\omega))\right )f_{\sigma^n \omega,0}\, dm\, d\mathbb P(\omega) \\
			&\phantom{=}+\varepsilon\int_\Omega \sum_{j=0}^{n-1}(\delta (\sigma^j \omega))^2\int_0^1 \int_0^1  (1-t) f_{\sigma^n \omega, 0}\mathcal L_{\sigma^{j+1}\omega, \varepsilon}^{n-j-1}(\partial^2_\gamma \mathcal L_{\beta(\sigma^j \omega)+t\varepsilon \delta (\sigma^j \omega)}(\mathcal L_\omega^j(f_{\omega, 0}h_\omega)))\, dm\, dt\, d\mathbb P(\omega) \\
			&=:D_{n,1}^1(\varepsilon)+D_{n, 2}^1(\varepsilon).
		\end{split}
	\end{equation}
	It follows from Lemmas~\ref{auxiliary} and~\ref{auxiliarly2} that
	\begin{equation}\label{fri1}
		|D_{n,2}^1(\varepsilon)| \le C_{\alpha, \underline{\alpha}, F} n|\varepsilon| \le C_{\alpha, \underline{\alpha}, F}n^{1-1/\gamma}.
	\end{equation}

	We now turn to $D_{n,1}^1(\varepsilon)$. By definition,
	\[
	\begin{split}
		D_{n,1}^1(\varepsilon)
		=\sum_{j=0}^{n-1}I_j, 
	\end{split}
	\]
	where 
	\[
	I_j:=  \int_\Omega \delta (\sigma^j \omega) \int_0^1 \mathcal L_{\sigma^{j+1}\omega, \varepsilon}^{n-j-1}\left (\partial_\gamma \mathcal L_{\beta(\sigma^j \omega)}(\mathcal L_\omega^j(f_{\omega, 0}h_\omega))\right )f_{\sigma^n \omega,0}\, dm\, d\mathbb P(\omega).
	\]
	For $j$ such that $n-j-1\ge \lfloor n/2\rfloor-1$, we have (using Lemma~\ref{auxiliary} and Theorem~\ref{DECC}) that 
	\begin{equation}\label{ij1}
		|I_j| \le C_{\alpha, F} (n-j-1)^{-1/\alpha+1}\le C_{\alpha, F} n^{-1/\alpha+1}.
	\end{equation}
	We now consider $j$ so that $j\ge \lfloor n/2\rfloor$. For $r>0$ to be fixed later, we have by writing 
	\[
	r^{-1}(\mathcal L_{\beta(\sigma^j \omega)+r}-\mathcal L_{\beta (\sigma^j \omega)})(\varphi)=\partial_\gamma \mathcal L_{\beta(\sigma^j \omega)}(\varphi)+r\int_0^1(1-t)\partial_\gamma^2 \mathcal L_{\beta(\sigma^j \omega)+tr}(\varphi)\, dt,
	\]
	and using Lemmas~\ref{auxiliary} and~\ref{auxiliarly2} 
	that
	\[
	\left \| \partial_\gamma \mathcal L_{\beta(\sigma^j \omega)}(\mathcal L_\omega^j(f_{\omega, 0}h_\omega))-r^{-1}(\mathcal L_{\beta(\sigma^j \omega)+r}-\mathcal L_{\beta (\sigma^j \omega)})(\mathcal L_\omega^j (f_{\omega, 0}h_\omega))\right \|_{L^1(m)} \le C_{\alpha, \underline{\alpha}, F}r.
	\]
	Writing $I_j$ as 
	\[
	\begin{split}
		&I_j=\\
		&=\int_\Omega \delta(\sigma^j \omega)\int_0^1 \mathcal L_{\sigma^{j+1}\omega, \ve }^{n-j-1}\left ( (\partial_\gamma \mathcal L_{\beta( \sigma^j \omega)}-r^{-1}(\mathcal L_{\beta(\sigma^j \omega)+r}-\mathcal L_{\beta (\sigma^j \omega)}))(\mathcal L_\omega^j (f_{\omega, 0}h_\omega))\right )f_{\sigma^n \omega, 0}\, dm\, d\mathbb P(\omega)\\
		&\phantom{=}+r^{-1}\int_\Omega \delta(\sigma^j \omega)\int_0^1 \mathcal L_{\sigma^{j+1}\omega, \ve}^{n-j-1}\left (((\mathcal L_{\beta(\sigma^j \omega)+r}-\mathcal L_{\beta (\sigma^j \omega)})(\mathcal L_\omega^j (f_{\omega, 0}h_\omega))\right )f_{\sigma^n \omega, 0}\, dm\, d\mathbb P(\omega),
	\end{split}
	\]
	we have that 
	\begin{equation}\label{upto}
		|I_j| \le C_{\alpha, \underline{\alpha}, F}r+C_{\alpha, F}r^{-1}j^{-1/\alpha+1}\le C_{\alpha, \underline{\alpha}, F}(r+r^{-1}n^{-1/\alpha+1}),
	\end{equation}
	where we used that 
	\[
	\|\mathcal L_{\sigma^{j+1}\omega, \ve}^{n-j-1}\left (((\mathcal L_{\beta(\sigma^j \omega)+r}-\mathcal L_{\beta (\sigma^j \omega)})(\mathcal L_\omega^j (f_{\omega, 0}h_\omega))\right )\|_{L^1(m)}\le 2\|\mathcal L_\omega^j (f_{\omega, 0}h_\omega)\|_{L^1(m)}\le 2C_{\alpha, F}j^{-1/\alpha+1}.
	\]
	Setting $r=n^{-\frac{1-\alpha}{2\alpha}}$, it follows that
	\begin{equation}\label{ij2}
		|I_j| \le C_{\alpha, \underline{\alpha}, F} n^{-\frac{1-\alpha}{2\alpha}}, \quad j\ge \lfloor n/2\rfloor.
	\end{equation}
	From~\eqref{ij1} and~\eqref{ij2} we conclude that
	\begin{equation}\label{fri3}
		|D_{n,1}^1(\ve)|\le \sum_{j=0}^{n-1}|I_j| \le C_{\alpha, \underline{\alpha}, F}n^{1-\frac{1-\alpha}{2\alpha}}.
	\end{equation}
	Estimate \eqref{eq:summability2-1} now follows readily from~\eqref{fri1}  and~\eqref{fri3}.
\end{proof}

$$
$$

Observe that the upper bound in Lemma \ref{summability1} 
is summable
if $\alpha < 1/4$, and the upper bound in Lemma \ref{summability2}
is summable if $\alpha < 1/5$. To complete the proof of 
Theorem \ref{diffvar}, it remains to prove the following:

\begin{lemma}\label{limitexists2}
	We have
	\[
	\lim_{\varepsilon \to 0}D_n^1(\varepsilon)=\sum_{j=0}^{n-1}\int_\Omega \delta (\sigma^j \omega)\int_0^1\mathcal L_{\sigma^{j+1}\omega}^{n-j-1} \left (\partial_\gamma \mathcal L_{\beta(\sigma^j \omega)}(\mathcal L_\omega^j(f_{\omega, 0}h_\omega))\right )f_{\sigma^n \omega, 0}\, dm\, d\mathbb P(\omega).
	\]
\end{lemma}

\begin{proof}
	We decompose as in \eqref{D12}:
	$$
	D_n^1(\varepsilon) = D_{n,1}^1(\varepsilon)+D_{n, 2}^1(\varepsilon),
	$$
	where $\lim_{\ve \to 0} D_{n, 2}^1(\varepsilon) = 0$ (see~\eqref{fri1}). Next,
	\begin{align*}
		D_{n,1}^1(\varepsilon) &=\int_\Omega\int_0^1 \sum_{j=0}^{n-1}\delta (\sigma^j \omega)\mathcal L_{\sigma^{j+1}\omega}^{n-j-1}\left (\partial_\gamma \mathcal L_{\beta(\sigma^j \omega)}(\mathcal L_\omega^j(f_{\omega, 0}h_\omega))\right )f_{\sigma^n \omega,0}\, dm\, d\mathbb P(\omega)\\
		&\phantom{=}+\int_\Omega\int_0^1 \sum_{j=0}^{n-1}\delta (\sigma^j \omega)(\mathcal L_{\sigma^{j+1}\omega, \varepsilon}^{n-j-1}-\mathcal L_{\sigma^{j+1}\omega}^{n-j-1})\left (\partial_\gamma \mathcal L_{\beta(\sigma^j \omega)}(\mathcal L_\omega^j(f_{\omega, 0}h_\omega))\right )f_{\sigma^n \omega,0}\, dm\, d\mathbb P(\omega)\\
		&=:\mathcal D_{n, 1}^1+\mathcal D_{n, 2}^1(\ve).
	\end{align*}
	Furthermore,
	\[
	\begin{split}
		& \mathcal D_{n, 2}^1(\ve)=\\
		&= \sum_{j=0}^{n-1}\sum_{k=0}^{n-j-2}\int_\Omega \delta(\sigma^j \omega) \int_0^1 \mathcal L_{\sigma^{j+k+2}\omega, \varepsilon}^{n-j-k-2}\int_{\beta(\sigma^{j+1}\omega)}^{\beta(\sigma^{j+1}\omega+\varepsilon \delta(\sigma^{(j+1)}\omega)}\partial_\gamma \mathcal L_\gamma (\mathcal L_{\sigma^{j+1}\omega}^k (\partial_\gamma \mathcal L_{\beta(\sigma^j \omega)}(\mathcal L_{\omega}^j(f_{\omega, 0}h_\omega))) \\
		&\phantom{=}\times f_{\sigma^n \omega, 0}\, d\gamma \, dm\, d\mathbb P(\omega).
	\end{split}
	\]
	Using Lemma~\ref{LEM12} together with Lemmas~\ref{auxiliary} and~\ref{auxiliarly3}, we have
	$$
	|\mathcal D_{n, 2}^1(\ve)| \le C_{\alpha, \underline{\alpha}, F}n^2|\varepsilon|.
	$$
	Hence, $\lim_{\ve \to 0} \mathcal D_{n, 2}^1(\ve) = 0$. This immediately implies the desired conclusion.
	
\end{proof}

\subsection{An improvement of Theorem~\ref{diffvar} for a special class of observables}
Here we indicate how we can relax the assumption $\alpha < 1/5$ in the statement of Theorem~\ref{diffvar} provided that we impose additional assumptions on $F$. We first note that the restriction $\alpha < 1/5$ is used only in Lemma~\ref{summability2}; the rest of the proof remains valid under the weaker condition $\alpha < 1/4$.

We continue to assume that $F\colon \Omega \times [0, 1]\to \mathbb R$ is measurable with $F(\omega, \cdot)\in C^2[0, 1]$ for $\omega \in \Omega$ and that~\eqref{observ2} is valid. We need the following
version of Lemma ~\ref{summability2}.

\begin{lemma}\label{improvedsum}
	Suppose that there exist $0 \le \eta \le \alpha$ and $C>0$ such that, for 
	$\bP$-a.e. $\omega \in \Omega$ and $x \in [0,1]$,
	\begin{align}\label{eq:special_F}
		|F_\omega(x)| \le C x^\eta \quad \text{and} \quad 
		\int_0^1 F_\omega h_\omega \, d m = 0.
	\end{align}
	Then,
	\begin{align}\label{eq:summability2-2}
		|D_n^1(\varepsilon)| \le C_{\alpha, \underline{\alpha}, F}(n^{ 1 - 1 / \gamma }+ 
		n^{ 1 - \frac{1 + \eta - \alpha }{ 2 \alpha } } ), \quad n<|\varepsilon|^{-\gamma}.
	\end{align}
\end{lemma}

\begin{proof}
	The proof proceeds as the proof of Lemma~\ref{summability2} up to~\eqref{upto}. It follows from \eqref{eq:fast_ml_2} that, for $\bP$-a.e. $\omega \in \Omega$,\
	$$ 
	\|\mathcal L_\omega^j (f_{\omega, 0}h_\omega)\|_{L^1(m)} \le   C_{\alpha, \underline{\alpha}, F} n^{ 1 - (1 + \eta)/\alpha}, \quad  j\ge \lfloor n/2\rfloor.
	$$
	Thus, choosing $r = n^{ \frac12 (  1 - \frac{1 + \eta }{\alpha} )  }$, we obtain 
	$$
	|I_j| \le C_{\alpha, \underline{\alpha}, F}(r+r^{-1}n^{ 1 - (1 + \eta)/\alpha }) \le 
	C_{\alpha, \underline{\alpha}, F} n^{ - \frac{1 + \eta - \alpha }{ 2 \alpha } }
	, \quad  j\ge \lfloor n/2\rfloor,
	$$	
	which yields \eqref{eq:summability2-2}.
	
\end{proof}

\begin{theorem}\label{thm:special}
	Suppose that~\eqref{eq:special_F} holds with $0\le \eta \le \alpha$ and $C>0$. Provided that $\alpha<\frac{1+\eta}{5}$, we have that $\varepsilon \mapsto \Sigma_\varepsilon^2$ is differentiable in $\varepsilon=0$.
\end{theorem}

\begin{proof}
	We proceed exactly as in the proof of Theorem~\ref{diffvar} with the only difference being that we use Lemma~\ref{improvedsum} instead of Lemma~\ref{summability2}. The conclusion follows since the upper bound~\eqref{eq:summability2-2} is summable if \(\alpha < \frac{1 + \eta}{5}\).
\end{proof}

\section*{Acknowledgements}
This paper has been funded by European Union – NextGenerationEU-Statistical properties of random dynamical systems and other contributions to mathematical analysis and probability theory (D. Dragi\v cevi\' c).  D.D  is grateful to Yeor Hafouta for the ideas in~\cite{DH2} used in the present paper.
J.L is supported by JSPS under the project LEADER. He would like to thank the University of Rijeka for hospitality over the course of this work.

\bigskip

\noindent\textbf{Competing interests:} the authors declare none.

\section{Appendix A}\label{AppA}

For the reader's convenience, we gather in this section preliminary results from \cite{DGTS} that 
are instrumental in the proofs of Theorems \ref{contvariance} and \ref{diffvar} concerning the regularity of 
the variance in the CLT.

Recall the definitions of the cones $\cC_* = \cC_*(a)$ and $\cC_2 = \cC_2(b_1, b_2)$ from Section \ref{sec:clt}. By 
\cite[Lemma 2 and Proposition 5]{DGTS}, there exist $a = a(\alpha) > 1$ and 
$b_1,b_2 > 0$,  such that for $0 < \beta \le \alpha$, we have
\begin{align}\label{cone_inv}
&\cL_{\beta}( \cC_* ) \subset \cC_*, \quad 
\cL_{\beta}( \cC_2 ) \subset \cC_2, \quad N_\beta( \cC_2 ) \subset \cC_2. 
\end{align}
The following follows from~\cite[Theorem 7]{DGTS}.
 \begin{theorem}\label{DECC}
There exists $C_\alpha>0$ with the property that for $\mathbb P$-a.e. $\omega \in \Omega$,  $n\in \mathbb N$, and $\phi, \psi \in \mathcal C_*$ such that $\int_0^1\phi\, dm=\int_0^1\psi\, dm$, we have
\begin{equation}\label{dec}
\int_0^1 | \mathcal L_\omega^n (\phi- \psi)|\, dm\le C_\alpha (\|\phi \|_{L^1(m)}+\|\psi \|_{L^1(m)})n^{-1/\alpha+1}.
\end{equation}
 \end{theorem}

Furthermore, by \cite[Lemma 12]{DGTS}, we have the following result.

\begin{lemma}\label{LEM12}
For $0<\gamma_0 \le \alpha<1$, there exist $0<\delta \le \alpha$  and $C_1, C_2>0$ such that 
\begin{equation}\label{lemm}
|(X_{\gamma}  N_{\gamma}(\phi))'(x)| \le C_1ax^{-\delta} \quad \text{and} \quad 
|(X_{\gamma}  N_{\gamma}(\phi))''(x)| \le C_2a x^{-\delta -1},
\end{equation}
for $x\in (0, 1]$, $\gamma \in [\gamma_0, \alpha]$ and $\phi \in \mathcal C_*(a)\cap \mathcal C_2(b_1, b_2)$ with $m(\phi)=1$. Moreover, there exist $\psi_i \in \mathcal C_*(a)\cap C^1(0, 1]$, $i\in \{1, 2\}$ such that 
\begin{equation}\label{psi12}
(X_\gamma N_\gamma (\phi))'=\psi_1-\psi_2,
\end{equation}
with $\int_0^1 \psi_1 \, dm = \int_0^1 \psi_2\, dm$
and $\| \psi_i\|_{L^1(m)}\le D$, $i \in \{1, 2\}$ for some $D= D_{\gamma_0, \alpha}>0$.
\end{lemma}

The following result on statistical stability is a consequence of \cite[Theorem 13]{DGTS}.

	\begin{proposition}\label{statstab}
	Let $h_{\omega, \ve}$ be defined as in Section \ref{sec:varcont} with $h_{\omega, 0} = h_\omega$. 
	There exists $C_\alpha>0$ such that for $\mathbb P$-a-e. $\omega \in \Omega$,
	\[
	\|h_{\omega, \varepsilon}-h_\omega \|_{L^1(m)} \le C_\alpha |\varepsilon|, \quad \varepsilon \in (-\varepsilon_0, \varepsilon_0).
	\]
\end{proposition}

\section{Appendix B}
\begin{lemma}\label{auxiliary}
	Let $\varphi \in C^2$ and $h\in \mathcal C_*\cap \mathcal C_2$ with $\int_0^1 h\, dm=1$. Then, there exist $\psi_i\in \mathcal C_*\cap \mathcal C_2$, $i=1, 2$ such that 
	\begin{equation}\label{difference}
		\varphi h=\psi_1-\psi_2.
	\end{equation}
	Moreover, $\|\psi_i\|_{L^1(m)}\le D\|\varphi\|_{C^2}$ for $i=1, 2$,  where $D$ depends only on $a$, $\alpha, b_1$ and $b_2$ (parameters of cones). 
\end{lemma}
\begin{proof}
	We recall the construction from~\cite{NPT}. Namely,
	take $\psi_1=(\varphi+\lambda x+A)h+B$ and $\psi_2=(\lambda x+A)h+B$. Observe that~\eqref{difference} holds. Moreover, it is proved in~\cite{NPT} that $\psi_1\in \mathcal C_*$ provided that:
	\begin{itemize}
		\item $\lambda=-2\|\varphi\|_{C^1}$ (so that $|\lambda|>\|\varphi'\|_{\infty}$ and $\lambda<0$);
		\item $A=3\|\varphi\|_{C^1}$ so that $\varphi+\lambda x+A\ge 0$;
		\item $B\ge \frac{a}{\alpha+1}\|\varphi\|_{C^1}$ so that 
		\[
		(\alpha+1)x^\alpha B\ge (\alpha+2)\lambda x^{\alpha+1}h(x)+x^{\alpha+1}h(x)\varphi'(x),
		\]
		where we also recall that $\lambda<0$;
		\item $B\ge \frac{4a}{a-1}\|\varphi\|_{C^1}$, so that 
		\[
		B\ge \frac{a}{a-1}[\sup(\varphi +\lambda x+A)]-\inf (\varphi+\lambda x+A)]
		\]
	\end{itemize}
	
	Note that 
	\[
	\psi_1'(x)=(\varphi'(x)+\lambda)h(x)+(\varphi(x)+\lambda x+A)h'(x).
	\]
	Consequently, 
	\[
	\begin{split}
		|\psi_1'(x)| &\le (\|\varphi'\|_\infty+|\lambda|)h(x)+ (\varphi(x)+\lambda x+A)|h'(x)| \\
		&\le a(\|\varphi'\|_\infty+|\lambda|)x^{-\alpha}+\frac{b_1}{x}(\varphi(x)+\lambda x+A)h(x) \\
		&\le a(\|\varphi'\|_\infty+|\lambda|)x^{-1}+\frac{b_1}{x}(\varphi(x)+\lambda x+A)h(x)\\
		&\le B \frac{b_1}{x}+\frac{b_1}{x}(\varphi(x)+\lambda x+A)h(x)\\
		&=\frac{b_1}{x}\psi_1(x),
	\end{split}
	\]
	provided that $B\ge \frac{3a}{b_1}\|\varphi\|_{C^1}$.
	
	Next, note that 
	\[
	\psi_1''(x)=\varphi''(x)h(x)+2(\varphi'(x)+\lambda)h'(x)+(\varphi(x)+\lambda x+A)h''(x).
	\]
	Thus, 
	\[
	\begin{split}
		|\psi_1''(x)| &\le \|\varphi''\|_\infty h(x)+2\frac{b_1}{x}(\|\varphi'\|_\infty+|\lambda|)h(x)+\frac{b_2}{x^2}(\varphi(x)+\lambda x+A)h(x) \\
		&\le a\|\varphi''\|_\infty x^{-2}+2ab_1(\|\varphi'\|_\infty+|\lambda|)x^{-2}+\frac{b_2}{x^2}(\varphi(x)+\lambda x+A)h(x)\\
		&\le B\frac{b_2}{x^2}+\frac{b_2}{x^2}(\varphi(x)+\lambda x+A)h(x)\\
		&=\frac{b_2}{x^2}\psi_1(x),
	\end{split}
	\]
	provided that $B\ge \frac{a}{b_2}\|\varphi''\|_\infty+\frac{6ab_1}{b_2}\|\varphi\|_{C^1}$. 
	Taking into account, various conditions for $B$,  we observe that we can choose $B$ of the form $B=d\|\varphi\|_{C^2}$, where constant $d>0$ depends only on $\alpha, a, b_1$ and $b_2$ to fulfill these. With this choice we have that 
	$\psi_1\in \mathcal C_2$.
	Finally, we observe that 
	\[
	\|\psi_1\|_{L^1(m)}\le \|\varphi\|_{\infty}+|\lambda|+A+B,
	\]
	which readily implies the desired conclusion.
\end{proof}

\begin{lemma}\label{auxiliarly2}
	Let $0<\underline{\alpha}\le \alpha <1$. Then, there exist $\delta=\delta(\underline{\alpha}, \alpha)\in (0, \alpha]$ such that
	\[
	|\partial_\gamma^2 \mathcal L_\gamma(\phi)(x)| \le C_{\alpha, \underline{\alpha}}x^{-\delta} \|\phi\|_{L^1(m)},
	\]
	for any $x\in (0, 1]$, $\gamma  \in [\underline{\alpha}, \alpha]$ and
	$\phi \in \mathcal C_*\cap \mathcal C_2$. In particular, 
	\[
	\|\partial_\gamma^2 \mathcal L_\gamma(\phi)\|_{L^1(m)}\le C_{\alpha, \underline{\alpha}}\|\phi\|_{L^1(m)}.
	\]
\end{lemma}

\begin{proof}
	The desired conclusion is obtained in the proof of~\cite[Lemma 16]{DGTS}, that is, in the arguments that yield the first estimate in~\cite[(27)]{DGTS}. We note that these arguments do not use the assumption that $\phi$ belongs to the cone $\mathcal C_3$ made in the statement of~\cite[Lemma 16]{DGTS}.
\end{proof}

\begin{lemma}\label{auxiliarly3}
	Let $0<\underline{\alpha}\le \alpha <1$. Then, there exist $\delta=\delta(\underline{\alpha}, \alpha)\in (0, \alpha]$ such that 
	\[
	|\partial_\gamma \mathcal L_\gamma (\phi) (x)|\le C_{\alpha, \underline{\alpha}}x^{-\delta}\|\phi\|_{L^1(m)},
	\]
	for $x\in (0, 1]$, $\phi \in C^1(0, 1]\cap \mathcal C_*$ and $\gamma \in [\underline{\alpha}, \alpha]$. In particular, 
	\[
	\|\partial_\gamma \mathcal L_\gamma (\phi)\|_{L^1(m)}\le C_{\alpha, \underline{\alpha}}\|\phi\|_{L^1(m)}.
	\]
\end{lemma}

\begin{proof}
	The proof proceeds as the proof of~\cite[Lemma 12]{DGTS} for verification of the first estimate in~\cite[(15)]{DGTS}. The only difference concerns the term $(N_\gamma (\phi))'$, which in~\cite{DGTS} was handled by relying on the assumption that $\phi \in \mathcal C_2$.
	
	In the present context, assuming without loss of generality that $\int_0^1\phi \, dm=1$, we have $N_\gamma (\phi)\in \mathcal C_*$ (since $\phi \in \mathcal C_*$). Moreover, $N_\gamma (\phi)\in C^1(0, 1]$ as $\phi \in C^1(0, 1]$. Consequently, \cite[Remark 1]{DGTS} implies that 
	\[
	|(N_\gamma (\phi))'(x)|\le \frac{\alpha+1}{x}N_\gamma (\phi)(x)\le C_\alpha x^{-1-\alpha},
	\]
	which is analogous to~\cite[(18)]{DGTS}. 
\end{proof}

\section{Appendix C}\label{AppB}
Here we provide the proof of Theorem~\ref{sharpl}. We stress that we closely follow the proof of~\cite[Theorem 14]{DGTS}.

\begin{proof}[Proof of Theorem~\ref{sharpl}]
	Using Theorem~\ref{DECC} and Lemma~\ref{LEM12}, it can easily be shown that the series in~\eqref{series} converges in $L^1(m)$ for $\mathbb P$-a.e. $\omega \in \Omega$ and that~\eqref{series2} holds.
	
	For $\mathbb P$-a.e. $\omega \in \Omega$ and each $\varepsilon$, we have 
	\[
	\begin{split}
		h_{\omega, \varepsilon}-h_\omega &=\mathcal L_{\sigma^{-l}\omega, \varepsilon}^l h_{\sigma^{-l}\omega, \varepsilon}-\mathcal L_{\sigma^{-l}\omega}^l h_{\sigma^{-l}\omega} \\
		&=\mathcal L_{\sigma^{-l}\omega, \varepsilon}^l 1-\mathcal L_{\sigma^{-l}\omega}^l 1 +\mathcal L_{\sigma^{-l}\omega, \varepsilon}^l(h_{\sigma^{-l}\omega, \varepsilon}-1)+\mathcal L_{\sigma^{-l}\omega}^l(1-h_{\sigma^{-l}\omega}),
	\end{split}
	\]
	for $l\in \mathbb N$. Hence, Theorem~\ref{DECC} implies that 
	\[
	\| h_{\omega, \varepsilon}-h_\omega -(\mathcal L_{\sigma^{-l}\omega, \varepsilon}^l1-\mathcal L_{\sigma^{-l}\omega}^l 1)\|_{L^1(m)} \le C_\alpha l^{1-1/\alpha}.
	\] From the above it follows that there exists $l_0=l_0(\varepsilon)\in \mathbb N$ such that for $l\ge l_0$,
	\begin{equation}\label{ap1}
		\left \| \frac{h_{\omega, \varepsilon}-h_\omega-\varepsilon\hat h_\omega }{\varepsilon}- \frac{\mathcal L_{\sigma^{-l}\omega, \varepsilon}^l1-\mathcal L_{\sigma^{-l}\omega}^l 1-\varepsilon \hat h_\omega}{\varepsilon}\right \|_{L^1(m)} \le C_\alpha |\varepsilon|.
	\end{equation}
	On the other hand,
	\[
	\mathcal L_{\sigma^{-l}\omega, \varepsilon}^l 1- \mathcal L_{\sigma^{-l}\omega}^l1=\sum_{j=0}^{l-1} \mathcal L_{\sigma^{-j}\omega, \varepsilon}^j (\mathcal L_{\sigma^{-(j+1)}\omega, \varepsilon}-\mathcal L_{\sigma^{-(j+1)}\omega})\mathcal L_{\sigma^{-l}\omega}^{l-1-j}(1).
	\]
	Writing 
	\[
	\begin{split}
		&(\mathcal L_{\sigma^{-(j+1)}\omega, \varepsilon}-\mathcal L_{\sigma^{-(j+1)}\omega})(\varphi)(x)\\
		&=\varepsilon \delta(\sigma^{-(j+1)}\omega)\partial_\gamma \mathcal L_{\beta(\sigma^{-(j+1)}\omega)} (\varphi)(x)\\
		&\phantom{=}+\varepsilon^2 (\delta(\sigma^{-(j+1)}\omega))^2\int_0^1 (1-t)\partial_\gamma^2 \mathcal L_{\beta(\sigma^{-(j+1)}\omega)+t\epsilon \delta(\sigma^{-(j+1)}\omega)}(\varphi)(x)\, dt,
	\end{split}
	\]
	we see that 
	\[
	\begin{split}
		&\frac{\mathcal L_{\sigma^{-l}\omega, \varepsilon}^l1-\mathcal L_{\sigma^{-l}\omega}^l 1-\varepsilon \hat h_\omega}{\varepsilon}\\
		&=\sum_{j=0}^{l-1}\delta (\sigma^{-(j+1)}\omega)\mathcal L_{\sigma^{-j}\omega, \varepsilon}^{j}\partial_\gamma \mathcal L_{\beta(\sigma^{-(j+1)}\omega)}(\mathcal L_{\sigma^{-l}\omega}^{l-1-j}(1))-\hat h_\omega \\
		&\phantom{=}+\varepsilon \sum_{j=0}^{l-1}(\delta(\sigma^{-(j+1)}\omega))^2 \int_0^1 (1-t)\mathcal L_{\sigma^{-j}\omega, \varepsilon}^j \partial_\gamma^2 \mathcal L_{\beta(\sigma^{-(j+1)}\omega)+t\epsilon \delta(\sigma^{-(j+1)}\omega)}(\mathcal L_{\sigma^{-l}\omega}^{l-1-j}(1))\, dt.
	\end{split}
	\]
	By Theorem~\ref{DECC} and~\cite[Lemma 16]{DGTS}, we have that 
	\begin{equation}\label{ap2}
		\begin{split}
			&\left \|\varepsilon \sum_{j=0}^{l-1}(\delta(\sigma^{-(j+1)}\omega))^2 \int_0^1 (1-t)\mathcal L_{\sigma^{-j}\omega, \varepsilon}^j \partial_\gamma^2 \mathcal L_{\beta(\sigma^{-(j+1)}\omega)+t\epsilon \delta(\sigma^{-(j+1)}\omega)}(\mathcal L_{\sigma^{-l}\omega}^{l-1-j}(1))\, dt\right \|_{L^1(m)}  \\
			&\le C_{\alpha, \underline{\alpha}}|\varepsilon|.
		\end{split}
	\end{equation}
	On the other hand, similarly to~\cite[Remark 21]{DGTS} for $n=\lfloor |\varepsilon|^{-\alpha}\rfloor$ and $l>n$,
	\begin{equation}\label{ap3}
		\begin{split}
			&\left \|\sum_{j=0}^{l-1}\delta (\sigma^{-(j+1)}\omega)\mathcal L_{\sigma^{-j}\omega, \varepsilon}^{j}\partial_\gamma \mathcal L_{\beta(\sigma^{-(j+1)}\omega)}(\mathcal L_{\sigma^{-l}\omega}^{l-1-j}(1))-\hat h_\omega\right \|_{L^1(m)} \\
			&\le \left \| \sum_{j=0}^n \delta  (\sigma^{-(j+1)}\omega) (\mathcal L_{\sigma^{-j}\omega, \varepsilon}^j-\mathcal L_{\sigma^{-j}\omega}^j)[(X_{\beta(\sigma^{-(j+1)}\omega)} N_{\beta(\sigma^{-(j+1)}\omega)}(\mathcal L_{\sigma^{-l}\omega}^{l-1-j}(1)))'] \right \|_{L^1(m)}\\
			&\phantom{\le }+
			\bigg \|\sum_{j=0}^n \delta (\sigma^{-(j+1)}\omega)\mathcal L_{\sigma^{-j}\omega}^j ((X_{\beta(\sigma^{-(j+1)}\omega)} N_{\beta(\sigma^{-(j+1)}\omega)}(\mathcal L_{\sigma^{-l}\omega}^{l-1-j}(1)))'- \\
			&\phantom{=}-(X_{\beta(\sigma^{-(j+1)}\omega)} N_{\beta(\sigma^{-(j+1)}\omega)} (h(\sigma^{-(j+1)}\omega)))')\bigg \|_{L^1(m)}+ C_\alpha |\varepsilon|^{1-2\alpha}.
		\end{split}
	\end{equation}
	Moreover, 
	\begin{equation}\label{ap4}
		\begin{split}
			&\left \| \sum_{j=0}^n \delta  (\sigma^{-(j+1)}\omega) (\mathcal L_{\sigma^{-j}\omega, \varepsilon}^j-\mathcal L_{\sigma^{-j}\omega}^j)[(X_{\beta(\sigma^{-(j+1)}\omega)} N_{\beta(\sigma^{-(j+1)}\omega)}(\mathcal L_{\sigma^{-l}\omega}^{l-1-j}(1)))']\right \|_{L^1(m)} \\
			&\le \sum_{j=1}^n \|(\mathcal L_{\sigma^{-j}\omega, \varepsilon}^j-\mathcal L_{\sigma^{-j}\omega}^j)[(X_{\beta(\sigma^{-(j+1)}\omega)} N_{\beta(\sigma^{-(j+1)}\omega)}(\mathcal L_{\sigma^{-l}\omega}^{l-1-j}(1)))']\|_{L^1(m)} \\
			&\le C_{\alpha, \underline{\alpha}}n^2|\varepsilon|\le C_{\alpha, \underline{\alpha}}|\varepsilon|^{1-2\alpha},
		\end{split}
	\end{equation}
	where we used that in the proof of~\cite[Theorem 14]{DGTS} it is showed that 
	\[
	\|(\mathcal L_{\sigma^{-j}\omega, \varepsilon}^j-\mathcal L_{\sigma^{-j}\omega}^j)[(X_{\beta(\sigma^{-(j+1)}\omega)} N_{\beta(\sigma^{-(j+1)}\omega)}(\mathcal L_{\sigma^{-l}\omega}^{l-1-j}(1)))']\|_{L^1(m)}\le C_{\alpha, \underline{\alpha}} n|\varepsilon|,
	\]
	for each $1\le j\le n$. Next,
	\begin{equation}\label{ap5}
		\begin{split}
			&\bigg \|\sum_{j=0}^n \delta (\sigma^{-(j+1)}\omega)\mathcal L_{\sigma^{-j}\omega}^j ((X_{\beta(\sigma^{-(j+1)}\omega)} N_{\beta(\sigma^{-(j+1)}\omega)}(\mathcal L_{\sigma^{-l}\omega}^{l-1-j}(1)))'- \\
			&\phantom{=}-(X_{\beta(\sigma^{-(j+1)}\omega)} N_{\beta(\sigma^{-(j+1)}\omega)} (h(\sigma^{-(j+1)}\omega)))')\bigg \|_{L^1(m)}\\
			&\le \sum_{j=0}^n \|(X_{\beta(\sigma^{-(j+1)}\omega)} N_{\beta(\sigma^{-(j+1)}\omega)}(\mathcal L_{\sigma^{-l}\omega}^{l-1-j}(1)))'- (X_{\beta(\sigma^{-(j+1)}\omega)} N_{\beta(\sigma^{-(j+1)}\omega)} (h(\sigma^{-(j+1)}\omega)))'\|_{L^1(m)}\\
			&\le n|\varepsilon|=O(|\varepsilon|^{1-\alpha}), 
		\end{split}
	\end{equation}
	where we used (see~\cite[Remark 21]{DGTS}) that (by making $l$ sufficiently large)
	\[
	\|(X_{\beta(\sigma^{-(j+1)}\omega)} N_{\beta(\sigma^{-(j+1)}\omega)}(\mathcal L_{\sigma^{-l}\omega}^{l-1-j}(1)))'- (X_{\beta(\sigma^{-(j+1)}\omega)} N_{\beta(\sigma^{-(j+1)}\omega)} (h(\sigma^{-(j+1)}\omega)))'\|_{L^1(m)} \le |\varepsilon|,
	\]
	for $0\le j \le n$. The conclusion of the theorem follows by putting \eqref{ap1}-\eqref{ap5} together.

\end{proof}

\end{document}